%% file: 2018-SEP-type-II-singularity.tex
\documentclass[reqno]{amsart}
\usepackage{amsmath,amssymb,amsthm,amsfonts,graphicx,color}
  \graphicspath{{../../drawings/}}
\definecolor{limegreen}{rgb}{0.196,0.804,0.196}
\definecolor{darkgreen}{rgb}{0.0,0.5,0.0}
\definecolor{darkbluegreen}{rgb}{0,0.3,0.6}
\definecolor{badgerred}{rgb}{0.715,0.004,0.004}
\usepackage[font=small,labelfont={bf,sf}, textfont={sf},margin=0em]{caption}
\usepackage{url,hyperref}

\usepackage{verbatim}

\usepackage{graphicx}

\usepackage{amsmath, amssymb, amsfonts, euscript, enumerate}
\usepackage{color}
\newcommand{\R}{{\mathbb R}}

\newcommand{\be}{\begin{equation}}
\newcommand{\ee}{\end{equation}}
\newcommand{\bee}{\begin{equation*}}
\newcommand{\eee}{\end{equation*}}
\newcommand{\fr} {\frac}
\newcommand{\e}{\epsilon}
\newcommand{\bw}{\bar w}
\newcommand{\ds}{\displaystyle}
\newcommand{\tw}{\hat w}
\newcommand{\bremark}{\begin{remark} \em}
\newcommand{\eremark}{\end{remark} }
\newtheorem{thm}{Theorem}[section]
\newtheorem{theorem}{Theorem}[section]
\newtheorem{lemma}[thm]{Lemma}

\newtheorem{claim}[thm]{Claim}
\newtheorem{prop}[thm]{Proposition}
\newtheorem{definition}[thm]{Definition}

\theoremstyle{remark}
\newtheorem{remark}{Remark}[section]
\numberwithin{equation}{section}

\title[Type II   Singularities on complete non-compact Yamabe flow ]{ Type II  Singularities on complete non-compact Yamabe flow}

\author{Beomjun Choi}
\address{ {\bf B. Choi:} Department of Mathematics, Columbia University, 2990 Broadway, New York, NY 10027, USA.}
\email{bc2491@columbia.edu }

\author{Panagiota Daskalopoulos}
\address{ {\bf P. Daskalopoulos:} Department of Mathematics, Columbia University, 2990 Broadway, New York, NY 10027, USA.}
\email{pdaskalo@math.columbia.edu}

\author{John King}
\address{ {\bf J. King:}  School of Mathematical Sciences, The University of Nottingham, University Park,
Nottingham, NG7 2RD}\email{john.king@nottingham.ac.uk}

\begin{document}


\begin{abstract}
This work concerns with the existence and detailed asymptotic analysis   of Type II singularities for solutions to  complete non-compact  conformally flat Yamabe flow with cylindrical behavior at infinity.   We provide the specific blow-up rate of 
the maximum curvature and show that the solution  converges,  after  blowing-up around the curvature maximum points, to  a rotationally symmetric steady soliton. It is the first time that the  steady soliton is shown to be a finite time singularity model of the Yamabe flow.  \end{abstract}
 \maketitle

\tableofcontents

\section{Introduction and preliminaries}

Let $(M,g_0)$ be a Riemannian  manifold without boundary  of dimension $n \geq 3$. For a mertric $$g = u^{\frac 4{n+2}} \, g_0$$
which is conformal to $g_0$, the scalar curvature $R$  of $g$ is given in terms of the
scalar curvature $R_0$ of $g_0$ by
$$R= u^{-1} \, \big ( - \bar c_n \Delta_{g_0} u^{\frac{n-2}{n+2}} + R_0 \, u^{\frac {n-2}{n+2}} \big )$$
where $\Delta_{g_0}$ denotes the Laplace Beltrami operator with respect to $g_0$ and $\bar c_n =  4 (n-1)/(n-2)$.

In 1989 R. Hamilton introduced the {\em Yamabe flow}
\begin{equation}
\label{eq-yamabe}
\frac{\partial g}{\partial t} = -R\, g
\end{equation}
as an approach to solve the {\em Yamabe problem}  on manifolds of positive conformal Yamabe invariant.
In the case where $M$ is compact the long time existence and convergence of Yamabe flow is well understood. 
Hamilton \cite{H} himself  showed the existence of the  normalized Yamabe flow (which is the re-parametrization of (\ref{eq-yamabe}) to keep the volume fixed)  
 for all time; moreover, in the case when the scalar curvature of the initial metric is negative,  he showed   the exponential convergence  of the  flow  to a  metric of constant scalar curvature. 
Chow \cite{Ch} showed the  convergence of the flow,  under the conditions that the initial metric is locally conformally flat and of positive Ricci curvature.
The convergence of the flow for any locally conformally flat initially metric was shown by Ye  \cite{Y}.

Schwetlick and Struwe  \cite{SS}  obtained the convergence
of the  Yamabe flow on a general compact manifold under
a suitable Kazdan-Warner type of condition that rules out the formation of bubbles
and that is verified  (via the positive mass Theorem) in dimensions $3 \leq n \leq 5$.
 The convergence  result,   in its full generality, was established by
 Brendle  \cite{S2} and \cite{S1}   (up to a technical assumption, in dimensions $n \ge 6$,   on the rate of vanishing of Weyl tensor at the points at which it vanishes):
starting with any smooth  metric on a compact manifold, the normalized Yamabe flow   converges to a metric of constant scalar curvature.

Although   the  Yamabe flow on compact manifolds is well understood,  the complete non-compact case is  unsettled. In this case one expects to
have more types of singularities which could  be either of type I or type II according to the definition below.

\begin{definition}
Assume that a solution $g(t)$ of the Yamabe flow \eqref{eq-yamabe} on a Riemannian  manifold  has a  singularity at time  $T$. 
This singularity is called type I  if $$ \limsup_{t\to T^-} \,\, (T-t)\, \sup_M |\mbox{Rm} |(\cdot,t)<+\infty.$$ A singularity which is not of type I, is called type II.  
\end{definition}

The results mentioned above show that in the   generic {\em  compact}  case the only singularities of the Yamabe flow are  type I. 
Moreover the works   \cite{DS1,  DKS}  address the  {\em singularity formation}  of  complete non-compact solutions to the conformally flat Yamabe flow whose conformal factors have {\em cylindrical behavior at infinity}. These  singularities are all of type I. 

A natural question to ask is whether the Yamabe flow {\em admits any singularities which are of type II in the non-compact case}. 
The authors in \cite{CD} presented,  for the first time in the Yamabe flow, examples of complete  solutions 
 which develop a   type II  singularity, either  at finite time  $T < +\infty$ or  at  infinite time  $T=+\infty$.  Such solutions are conformally equivalent to $\R^n$ and their initial data  has cylindrical behavior at infinity if $T<\infty$. 
What distinguishes our type II solutions from the type I solutions which are modeled on shrinkers, is that their initial metric has {\em  slower second order decay rate} to the cylindrical metric than that of any other Yamabe shrinkers. In this work we study complete non-compact and conformally flat solutions of the Yamabe flow \eqref{eq-yamabe} on $\R^n$ which develop {\em  type II singularity}  and provide  their {\em detailed  asymptotic behavior}  near the singularity.

Let us briefly discuss next the known results on the singularity formation of non-compact Yamabe flow. Even though the analogue of Perelman's monotonicity formula is still lacking for the Yamabe flow, one expects that  Yamabe soliton solutions  model finite and infinite time singularities. These are special solutions  of the Yamabe flow \eqref{eq-yamabe} which is characterized by a metric $g=g_{ij}$ and a {\em potential  function $P$}  so that
$$(R-\rho)g_{ij} = \nabla_i\nabla_j P, \qquad \rho \in \{1,-1,0\} .$$
Depending on the sign of the constant $\rho$,  a Yamabe soliton is  called a {\em Yamabe shrinker}, a {\em Yamabe expander} 
or a { \em Yamabe steady soliton}  if $\rho=1,-1$ or $0$ respectively. 

The classification of locally conformally flat Yamabe solitons
with positive  sectional curvature was  established in \cite{DS} (c.f. also \cite{CSZ} and \cite{CMM}). 
It is shown  in \cite{DS} that such solitons are globally conformally equivalent to $\R^n$ and   correspond to radially symmetric self-similar solutions of the fast-diffusion equation 
\begin{equation}\label{eq-flatyamabe}
u_t = \fr{n-1}{m}\,  \Delta \, u^{\frac  {n-2}{n+2}}, \quad \mbox{on}\,\, \R^n \times [0,T)
\end{equation}
satisfied by the conformal factor defined by $ g_{ij}=u^{\fr{4}{n+2}} \delta_{ij}  $.   Here and in the 
sequel $\delta_{ij}$  denotes the standard metric on $\R^n$ and we set ${\ds m:=({n-2})/({n+2})}$. A complete description of those solutions is  given in \cite{DS}. 
In \cite{CSZ} the assumption of positive sectional curvature was relaxed to that of nonnegative Ricci curvature.

As mentioned above, in   \cite{DS1,  DKS}   the  {\em singularity formation}  of  complete non-compact solutions to the conformally flat Yamabe flow with  {\em cylindrical behavior at infinity} was studied.  The singularity profiles  are Yamabe shrinking solitons which are determined  by the second order asymptotics at infinity of the initial data, which is  matched with that of the corresponding soliton. 
The  solutions  may become  extinct at the extinction
time $T$ of the cylindrical tail  or may live longer than $T$. In  the first case,  the singularity profile  is described by a {\em Yamabe shrinker}  that becomes extinct at time $T$. 
This result can be seen as a stability result around the Yamabe shrinkers with cylindrical behavior at infinity. 
In the second  case, the flow develops a singularity at time $T$  which is described by a {\em singular}  Yamabe shrinker slightly before $T$
and by  a matching  {\em Yamabe expander}  slightly after $T$. 

Recently the authors \cite{CD} studied long time behavior of the complete non-compact conformally flat Yamabe flow and in particular showed the stability around the steady solitons. Such solitons are conformally equivalent to $\R^n$ and rotationally symmetric. They are the analogue to the bowl translating soliton of MCF or the Bryant soliton  of the  Ricci flow. 

In this work we study  the  {\em  asymptotic behavior} of type II singularities in the conformally flat non-compact case.
More precisely,  for a sufficiently small  $T < +\infty$, 
we provide a  condition, in terms of the second order decay rate of  the initial metric $g_{\gamma,\kappa}(\cdot, 0)$
at spatial infinity, which guarantees that the Yamabe flow $g_{\gamma,\kappa}(\cdot,t)$ develops
a type II singularity at time $T$ with specified blow up rate 
\begin{equation}  \limsup_{t\to T^-} \,\,  (T-t)^{1+\gamma}\sup_M | \mbox{\em Rm} | \, (\cdot,t) =\kappa. \end{equation}  Moreover, 
we prove that after rescaling  the  solution  $g_{\gamma,\kappa}(\cdot, t) $  around highest curvature point  by $(T-t)^{-(1+\gamma)}$,
 it  converges to a radial steady gradient soliton.     
 
 Our main result states as follows:
 \begin{thm}\label{thm-main}
Let $g_0=u_0^{1-m}(x) \, \delta_{ij} $ be a conformally flat metric with positive Ricci curvature. For any given $\gamma>0$ and $A>0$, there is $T_1>0$ with the following property:  for any $T < T_1$, if   
\begin{enumerate}[{\em i)}]
\item $u_0^{1-m} (x)<\fr{(n-1)(n-2)}{|x|^2} \, T, \,\,\,  \forall  x \in \R^n$, and  
\item $ u_0^{1-m}(x) = \fr{(n-1)(n-2)}{|x|^2}\,  \left ( T- \left(\fr{\ln |x|}{A}\right)^{-\fr{1}{\gamma}} + O\big ((\ln |x|)^{-\fr{1}{\gamma}-1}\big)\right),
\,\, $ as  $|x| \to +\infty$
\end{enumerate}
 then the solution of Yamabe flow \eqref{eq-yamabe} with initial data $g_0$ will develop
a type II singularity at time $t=T$ with   specified blow up rate given by 
\begin{equation}\label{eq-blowuprate}  \limsup_{t\to T^-} \,\,  (T-t)^{1+\gamma}\sup_M | \mbox{\em Rm} | \, (\cdot,t) =\fr{2\gamma A}{\sqrt{ n(n-1)}}. \end{equation}  Moreover,  after rescaling the metric around the highest curvature point  by $(T-t)^{-(1+\gamma)}$,
it converges to the unique radial steady gradient soliton of maximum scalar curvature $2\gamma A$. 
\end{thm}
 
\smallskip

Theorem \ref{thm-main}  shows   for the first time that the  conformally flat radial steady soliton appears as a  finite time singularity model for  the Yamabe flow. 
In  the Mean curvature flow and the Ricci flow, examples of  type II singularities  and their asymptotic behavior  were shown in both compact and non-compact settings under  rotational symmetry (c.f. \cite{AV}, \cite{IW} for Mean curvature flow and \cite{AIK2},\cite{AIK},\cite{W} for Ricci flow).    Let us remark that unlike in the cases mentioned above our result also includes {\em non-radial initial data. }

\smallskip

To achieve our result  we first construct sharp barriers (super and sub-solutions) for given fixed $\gamma$ and $A$.
The barriers are chosen sufficiently close to each other so that they give a model solution whose blow up limit at the tip is a radial steady soliton with the curvature blow up rate \eqref{eq-blowuprate}. In other words this already proves the result of Theorem \ref{thm-main} if the initial metric is in between sharp barriers. When the initial metric satisfies the condition of Theorem \ref{thm-main}, the solution can be located between two model solutions which are differ by a translation in cylindrical direction. Then we do further analysis to show the solution also has the same asymptotic behavior and curvature blow up rate. 

The barriers will be constructed to be radially symmetric though we don't assume it for initial metrics. Note that for a radially symmetric solution $g:=u^{\frac 4{n+2}}(r,t)\, (dr^2+r^2dg_{S^{n-1}})$ of the Yamabe flow \eqref{eq-flatyamabe}, it is often convenient to work in  cylindrical coordinates where the metric is expressed as 
$g=w(s,t)\, (ds^2+dg_{_{S^{n-1}}})$,  with $s=\ln r$. The conformal factor in cylindrical coordinates is given by 
\begin{equation}\label{eq-uw}
w(s,t)=r^2 \, u^{\frac 4{n+2}}(r,t), \qquad s=\ln r
\end{equation}
and satisfies the equation 
\begin{equation}\label{eq-w}
\fr{m}{n-1}\Bigl(w^\fr{n+2}{4}\Bigr)_t=\Bigl(w^\fr{n-2}{4}\Bigr)_{ss} -\Bigl(\fr{n-2}{2}\Bigr)^2w^\fr{n-2}{4}, \qquad \mbox{on} \,\,   \R \times [0,T]
\end{equation}
with ${\ds m:=\frac{n-2}{n+2}}$. 

\medskip

The {\em outline of our paper}  is as follows: In the Section \ref{sec-formal}  we will begin by giving the {\em formal matched asymptotics}   of 
the  type II  singularity at time $t=T$. Based on this analysis,  we will  introduce in Section \ref{sec-notation}   the two different regions {\em outer and inner}  and the scalings
in each region.  Also some notation. We will refer to the notation of this  section throughout the paper.  Section \ref{sec-outer}  deals with  the   {\em barrier construction}  in the {\em outer}  region  (c.f. Propositons \ref{prop-outer} and  \ref{prop-outer2}) and  Section \ref{sec-inner} deals with the {\em barrier construction}   in the {\em inner}  region  (c.f. Proposition \ref{prop-inner}). 
Combining the results from Sections \ref{sec-outer} and
\ref{sec-inner}, in Section \ref{sec-barriers} we will glue the barriers in the
inner and outer region to construct suitable  {\em super and sub-solutions}. In Section \ref{sec-asymptoticshape} we will show one of our main results, Theorem \ref{thm-main2}, 
which shows the {\em type II convergence}  of any   given conformally flat Yamabe flow   to the steady soliton,  assuming  that its initial data is trapped between 
our super and sub-solutions. Finally, our last section \ref{sec-new} will be devoted to  the proof of Theorem \ref{thm-main}.  
In this section, along  with the barriers constructed in previous sections, we will make use  of the differential Harnack inequality  in \cite{Ch} and  the classification result of  Yamabe solitons  (c.f. \cite{DS,CMM,CSZ}).
\bigskip


\section{Formal matched asymptotics}\label{sec-formal}
In  this section we will present the formal construction of our solutions   which is based on matched asymptotic analysis. We hope that this will  give our  reader an intuition for our  construction. 

For any given parameters $\gamma>0$ and $A>0$  we will construct below a family of formal  rotationally symmetric  solutions where the 
curvature blows up in the type II rate 
\begin{equation}\label{eqn-cur1} 
\lim_{t \to T^-} \, (T-t)^{1+\gamma}\sup_{M_t}|  \mbox{Rm}  |=\fr{2\gamma A}{\sqrt{n(n-1)}}.
\end{equation}
Note that our main results Theorem \ref{thm-main1} and \ref{thm-main2} are  not restricted on rotationally symmetric solutions, however the 
barriers which will be constructed in next sections are obtained from perturbations of this formal solution and those are radial. 

Motivated by condition  \eqref{eqn-cur1} and in order to capture at the end a stationary solution, we perform the  change of variables 
on our solution $w(s,t)$ of \eqref{eq-w} setting
\begin{equation}\label{eqn-cv1}
\hat w(\eta, \tau)=  (T-t)^{-1} w (s,t), \quad  \eta=(T-t)^\gamma s, \,\, \tau = - \ln \, (T-t).
\end{equation}
 A direct calculation shows that $\hat w(\eta,\tau)$ satisfies the evolution 
\begin{equation}\label{eq-tildew}
B[\tw] =0
\ee
where 
\be\label{eqn-derror}
\begin{split}
B[\tw]  :=  \partial_\tau \hat w  - (n -1) &e^{-2\gamma \tau} \, \Bigl (  \fr{\hat w_{\eta\eta}}{\hat w}+\fr{n-6}{4}\fr{\hat w_\eta ^2}{\hat w ^2} \Bigr ) \\  &- \big (  \gamma \eta   \,   \hat w_\eta +  \hat w - (n-1)(n-2) \big ).
\end{split}
\ee
Thus, if we assume that   
\begin{equation}\label{eqn-lower1} 
\partial_\tau \hat w \qquad \mbox{and} \qquad  e^{-2\gamma \tau} \Bigl(\fr{\hat w_{\eta\eta}}{\hat w}+\fr{n-6}{4}\fr{\hat w_\eta ^2}{\hat w ^2} \Bigr )
\end{equation} 
are negligible as $\tau \to \infty$, the above  equation is  reduced to the following ODE in $\eta$ variable 
\begin{equation}\label{eq-reduced}
\gamma    \eta\,   \hat w_\eta + \hat w-(n-1)(n-2) =0.
\end{equation}
Solving  this equation on  $\eta >0$ gives the solution $\tw_0(\eta)$ given by 
\begin{equation}\label{eq-tildewansatz}
\hat w_0(\eta)  := (n-1)(n-2) \bigl (  1- \kappa \, \eta ^{-1/\gamma} \bigr )\end{equation}
for a parameter  $\kappa \in \R$. We will assume from now on, without loss of generality,  that $\kappa >0$,   although $\kappa\le 0$ also gives a solution. Moreover, there is a   family of solutions to \eqref{eq-reduced} on $\eta <0$, but  this case is  exactly symmetric to the $\eta>0$ case which we will handle below.  
Indeed, the solutions  given by \eqref{eq-tildewansatz}   describe non-compact surfaces  moving in positive $\eta$ (hence positive $s$) direction,  while the corresponding  solutions defined on $\eta<0$ describe a symmetric  surface just moving on  the opposite direction. 
\smallskip 
This ansatz, namely setting  
\begin{equation}\label{eqn-www}
\hat w(\eta,\tau) := (n-1)(n-2) \, \bigl (  1- \kappa\,  \eta ^{-\fr{1}{\gamma}} \bigr )
\quad \text{ on }\,\, \eta > \kappa^{\gamma}
\end{equation}
 approximates a solution  of the equation \eqref{eq-tildew}. In fact, plugging $\hat w(\eta,\tau)$ given by \eqref{eqn-www}  into \eqref{eq-tildew},  we see that the error term  
 becomes 
\be\label{eqn-bw}
B[\tw] = - (n-1) e^{-2\gamma \tau}  \, \Bigl (\fr{\hat w_{\eta\eta}}{\hat w}+\fr{n-6}{4}\fr{\hat w_\eta ^2}{\hat w ^2} \Bigr).
\ee
Writing ${\ds B[\tw]=(n-1)(\hat w e^{\gamma \tau} )^{-2} \big (\hat w \hat w _{\eta\eta}+ \fr{n-6}{4} {\hat w}_\eta^2 \big )}$, we see that it 
becomes  arbitrarily small in  the   space-time region 
$$(e^{\gamma \tau} \,  \hat w)^{-1}=o(1), \qquad  \,\, \mbox{as}\,\,  \tau  \to +\infty$$
which  we call  the  {\em  outer region} (see Figure \ref{figure0} below). This is  the region where the diffusion doesn't  play an
 important    role and the remaining advection and reaction terms are dominant. 

The  {\em inner-region}  is given by 
\be\label{eqn-dfn-inner}
 e^{\gamma \tau}\, \hat w(\eta,\tau) = O(1),  \qquad \mbox{as} \,\, \tau   \to +\infty
 \ee
(see Figure~\ref{figure0}). In this region  
we  perform another scaling, setting 
\begin{equation}\label{eqn-cv2}
\bar w(\xi,\tau)=e^{\gamma\tau}\hat w(\eta,\tau), \qquad \eta= A + e^{-\gamma\tau}\xi
 \end{equation}for some choice of $A>0$, which combined with  \eqref{eqn-cv1} gives 
\begin{equation}\bar w(\xi, \tau) = e^{\gamma \tau} \, \hat w (A+e^{-\gamma\tau}\xi,\tau) = e^{(1+\gamma)\tau} \, w(Ae^{\gamma\tau}+\xi,T-e^{-\tau}). \end{equation}
The evolution equation for $\bar w(\xi,\tau)$is given by 
\begin{equation}\label{eq-barw}\begin{split}
I[\bar w] :=  & e^{-\gamma \tau} \Bigl( \bar w_\tau- (1+\gamma) \bar w  \Bigr) -(n-1)\Bigl(\fr{\bar w_{\xi\xi}}{\bar w}+\fr{n-6}{4} \, \fr{\bar w_\xi ^2}{\bar w ^2} \Bigr ) 
\\ &+(n-1)(n-2) - \gamma A \bar w_\xi=0 \end{split} . \end{equation}
Thus assuming that in this region  the first term having $e^{-\gamma\tau}$ becomes  negligible  as $\tau\to\infty$, the equation is reduced to 
\begin{equation}\label{eq-barwapprox}\begin{aligned}
(n-1)\Bigl(\fr{\bar w_{\xi\xi}}{\bar w}+\fr{n-6}{4}\fr{\bar w_\xi ^2}{\bar w ^2} \Bigr)-(n-1)(n-2) + \gamma A \bar w_\xi  =0\end{aligned}. \end{equation} This can be seen as the equation satisfied by a traveling wave solution  of \eqref{eq-w} with speed $\gamma A$. For each $A>0$, this equation admits a $\tau$ independent entire solution  $\bar w(\xi)$ which is unique up to translation in $\xi$. For such  profile $\bar w$, the function   $\bar w (s-\gamma At)$ becomes   a traveling wave solution of  equation \eqref{eq-w}. 
From the geometric point of view,    $\bar w(s)$ corresponds to a radially symmetric non-compact metric on  $\R^n$ via \eqref{eq-uw} and solutions with different $A$'s are the only radially symmetric {\em steady gradient solitons} on locally conformally flat manifolds (c.f.  \cite{CSZ, DS}).  
For such a solution $\bar w$, the highest curvature point is at the origin (i.e.  $s=-\infty$) and one may formally compute that ${\ds |\operatorname{Rm}|_{\max}(t) = \fr{2\gamma A}{\sqrt{n(n-1)}}}$, hence leading to  \eqref{eqn-cur1}.   

\begin{remark}
Note that  a solution $\bar w(\xi,\tau)$   of \eqref{eq-barwapprox} which  also depends  on $\tau$, 
can be  written as \begin{equation}\bar w(\xi,\tau)=\bar w_0 (\xi + C(\tau)) \end{equation}
for a function $C(\tau)$, where  $\bar w_0$ is one $\tau-$independent solution of \eqref{eq-barwapprox}. By plugging this into \eqref{eq-barw} again, we get an error term \begin{equation}e^{-\gamma \tau} \, \Bigl( (1+\gamma)\bar w-C'(\tau)\,  \bar w_{\xi} \Bigr) \approx 0.\end{equation}
We  may choose  $C(\tau)$ so that $C'$ is small and thus the  error term above vanishes appropriately  as $\tau\to \infty$.  Later, we will use this $C(\tau)$ to glue barriers from the two different regions, inner and outer.  
\end{remark}
\begin{figure}
\def\svgwidth{\linewidth}{\tiny
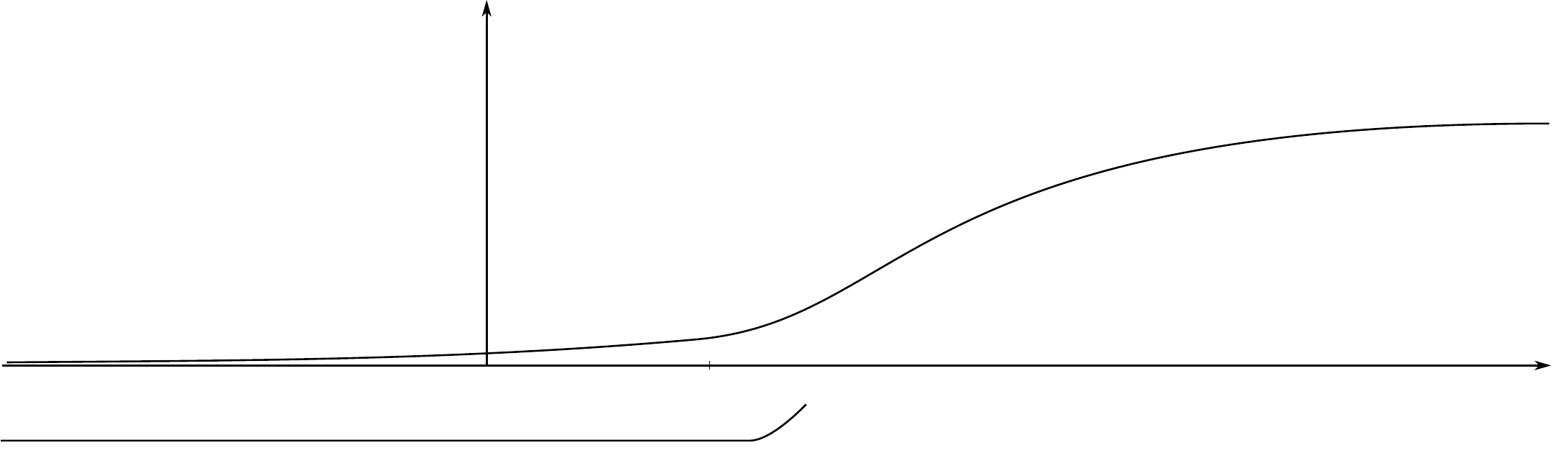
\caption{}\label{figure0}}

\end{figure}

\smallskip
Next, we will carry out a  matching asymptotic analysis between the inner and outer  regions and obtain a relation 
between $\kappa>0$ in \eqref{eq-tildewansatz} and $A>0$ in \eqref{eq-barw}.
It is known that a solution $\bar W(\xi)$ of \eqref{eq-barwapprox} satisfies the  asymptotic behavior 
$$\bar W(\xi)\approx \fr{(n-1)(n-2)}{\gamma A}\xi + O(1), \qquad  \mbox{ as }\,\,  \xi \to \infty.$$
Therefore, recalling that $\eta= A+e^{-\gamma \tau} \, \xi$,  our solution $\bar w(\xi,\tau)$ which is approximately $\bar W(\xi)$ satisfies  
$$
  e^{-\gamma\tau} \, \bar w(\xi,\tau) \approx  \fr{(n-1)(n-2)}{\gamma A}(\eta-A)+ o(1), \quad  
  (\eta-A) \, e^{\gamma\tau} \gg 1,\, \, \tau \gg1.
$$
 On the other hand, from the outer region ansatz \eqref{eq-tildewansatz}, by the first order Taylor approximation 
 near $\eta=\kappa^\gamma$ we have
 \begin{equation*}\begin{aligned}\hat w(\eta, \tau) &\approx (n-1)(n-2) \, \fr{\kappa}{\gamma} \, (\kappa^\gamma)^{-\fr{1}{\gamma}-1}(\eta-\kappa^\gamma)+o(1)\\ & \approx\fr{(n-1)(n-2)}{\gamma \kappa^\gamma}(\eta-\kappa^\gamma)+o(1), \qquad 
 \mbox{for}\, \, \eta \, \,  \mbox{near}\, \,   \kappa^\gamma, \,\,  \tau \gg1. \end{aligned} \end{equation*}
Thus, we  can see that these two asymptotics are matched exactly if $$A=\kappa^\gamma.$$
To see this in another way, we can argue that if  $A<\kappa^\gamma$   the  two asymptotics are inconsistent as $\eta\to (\kappa^\gamma)_+$  and if $\kappa^\gamma <A$  the linearization of  \eqref{eq-tildewansatz} near $\eta=A$ is inconsistent with the asymptotic behavior from the inner region.

\section{Notation and different scalings}\label{sec-notation}  
In this section we will  summarize the coordinates and  different scalings of our solutions, as introduced in the previous
section. We will refer to the notation introduced below  throughout the paper.

\subsection{Coordinate systems} 

Since our metric $g(t)$ is conformally equivalent to the standard euclidean metric on $\R^n$ denoted briefly by $\delta_{ij} $,    we represent $$g(t)=u^{1-m}(x,t) \, \delta_{ij} $$ 
where the scalar function $u(x,t)$ evolves by  equation \eqref{eq-flatyamabe} under  Yamabe flow. 

In case when the metric is radial, it is often convenient to work in the cylindrical coordinates, that is 
$$ g(t)=u^{1-m}(r,t) (dr^2 + r^2 g_{S^{n-1}}) = w(s,t) (ds^2 + g_{S^{n-1}}), \quad r=|x|$$ and $$w(s,t)=r^2 u^{1-m}(r,x), \qquad s=\ln r=\ln |x|.$$
Under this coordinate change, $w(s,t)$ evolves by  equation \eqref{eq-w}.

\subsection{Scalings} We use the following different scaling in different regions: 

\begin{itemize} 

\item In the {\em outer region}, the conformal factor is represented by $\tw(\eta,\tau)$ and is scaled from $w(s,t)$ as follows 
$$\hat w(\eta,\tau)=e^{\tau}w(e^{\gamma\tau}\eta,T-e^{-\tau}).$$
The function $\tw (\eta,\tau)$ evolves by the equation \eqref{eq-tildew}.

\item In the {\em inner region}, the conformal factor is represented by $\bar w(\xi,\tau)$ and is scaled from previous factors as follows 
$$\bar w(\xi,\tau) = e^{\gamma \tau} \tw (A+\xi e^{-\gamma \tau} ,\tau)= e^{(1+\gamma)\tau}w (Ae^{\gamma\tau} + \xi,T-e^{-\tau}).$$
The function $\bar w(\xi, \tau)$ evolves by the equation \eqref{eq-barw}.

\item The above scaling change from  $w(s,t)$ to $\bar w(\xi,\tau)$ corresponds  to  the following scaling change in euclidean coordinates  from  $u^{1-m}(x,t)$ to  $\bar u^{1-m} (y,l)$:
\begin{equation}\label{eqn-u10}
|x|^2\, u^{1-m}(x,t)=(T-t)^{1+\gamma}\,  |y|^2 \, \bar u^{1-m}(y,l) \end{equation}
where $l$ is a new time variable $l= (T-t)^{-\gamma}/\gamma$ and 
\begin{equation}\label{eqn-u11}
x= y \, e^{\gamma A l}= y \, e^{A (T-t)^{-\gamma}}.
\end{equation}
This scaling is used only in Section \ref{sec-asymptoticshape}.  $\bar u(y,l)$ evolves by equation \eqref{eqn-ry1}.

\end{itemize} 
\subsection{Relations}
Let us summarize relations between different variables and functions appearing  in the  scalings introduced above. 

\begin{itemize}
\item We use {\em three time scales}:
\begin{equation}
t\in[0,T),\quad \tau=-\ln (T-t),\quad  l=\fr{(T-t)^{-\gamma}}{\gamma}=\fr{e^{\gamma \tau}}{\gamma}.
\end{equation}
The last one will only be used in the last section.

\item The  {\em three space   scales}  in   cylindrical  coordinates are: 
\begin{equation}\label{eqn-s10}
s \in \R, \quad \eta: =e^{-\gamma\tau}  s, \quad  \xi=s - A \, e^{\gamma\tau}=  (\eta-A)\, e^{\gamma \tau}.
\end{equation}

\item The  corresponding {\em three  conformal factors}  in cylindrical coordinates at the different scales are: 
\begin{equation}\label{eqn-w10}
w(s,t)= e^{-\gamma \tau} \, \hat w (\eta,\tau)= e^{-(1+\gamma)\tau} \, \bar w (\xi,\tau). \end{equation}

\end{itemize} 

In the {\em radial case}, $w(\xi,\tau)$ and $\bar u^{1-m}(y,l)$      given by \eqref{eqn-u10} are related by 
\be\label{eqn-uw10}
 \bar w (\xi,\tau)=  |y|^2 \, \bar u(y,l), \qquad \xi=\ln |y|, \qquad  l=\fr{e^{\gamma \tau}}{\gamma}.
 \ee

\subsection{Functions}
We introduce below the  functions $\tw_0$, $\bar w_0$, $\bar U$, which play  important roles in the paper.
 
\begin{itemize}
\item For every  $A>0$, we define $\hat w_0(\eta)$ to be  the {\em outer region ansatz } 
 \begin{equation}\hat w_0(\eta) := (n-1)(n-2) \, \Bigl[ 1- \Bigl(\fr{\eta}{A}\Bigr) ^{-\fr{1}{\gamma}} \Bigr]
 \text{ on }\eta > A,\end{equation} which is a solution of \eqref{eq-reduced}. 
 
 \item For the same $A>0$ as above, we  denote by $\bar w_0(\xi)$ the 
 {\em 
 inner region approximation} which  the solution of   equation 
 \begin{equation}\begin{aligned}\label{eqn-soliton1} 
(n-1)\Bigl(\fr{\bar w_{\xi\xi}}{\bar w}+\fr{n-6}{4}\fr{\bar w_\xi ^2}{\bar w ^2} \Bigr)-(n-1)(n-2) + \gamma A \, \bar w_\xi  =0.\end{aligned}\end{equation}
 having asymptotic behavior  
\be\label{eqn-soliton2}
\bar w_0 (\xi) = \fr{(n-1)(n-2)}{\gamma A} \xi + 0 + \fr{(n-1)(n-6)}{4\gamma A} \fr{1}{\xi}+ O\Bigl(\ \fr{1}{\xi^2}\Bigr).
\ee
This solution represents a steady gradient soliton of the flow and it is unique up to translation in $\xi$.

\item Finally  $\bar U$ denotes the representation of $\bar w_0$ on $\R^n$ by the following change of coordinate \be|x|^2 \, \bar U^{1-m}(x)= \bar w_0(\ln |x|).\ee

\end{itemize}

\section{Barrier construction in the outer region}\label{sec-outer}

Let us fix parameters $\gamma >0$ and $A >0$ as they  appear in the curvature blow up rate of our  singularity  \eqref{eqn-cur1}. 
In this section we will construct  appropriate  super and  sub solutions in the {\em outer region} 
$(e^{\gamma \tau} \,  \hat w (\eta,\tau))^{-1}=o(1), \qquad  \,\, \mbox{as}\,\,  \tau  \to +\infty$, which will be  given by
$$\{ (\eta,\tau) \, | \, \eta \ge A +  \xi_0 \, e^{-\gamma \tau}, \,\, \tau\ge\tau_0 \}$$
for some $\xi_0 >0$.  

 Recall that for a rotationally symmetric solution $u(r,t)$, $r=|x|$  of the conformally flat Yamabe flow \eqref{eq-flatyamabe}, we perform the  cylindrical change of coordinates \eqref{eq-uw} leading to a solution  $w(s,t)$ of \eqref{eq-w}. 
As already seen in section \ref{sec-formal}, to capture the behavior in the outer region we perform a further change of variables \eqref{eqn-cv1} leading to a solution $\hat w(\eta,\tau)$ of \eqref{eq-tildew}. Assuming that near our singularity \eqref{eqn-lower1} holds, we find that the zero  order behavior of $\hat w$ near the singularity in the outer region is given by a  solution of the ODE
\eqref{eq-reduced}. The general solution of \eqref{eq-reduced}  is given by \eqref{eq-tildewansatz} for a parameter $\kappa >0$. Thus, setting   $A:=\kappa^{1/\gamma} >0$   we define the zero order
approximation of $\hat w$ in the outer region to be  
%
%
\begin{equation}\label{eq-tildew0}\hat w_0(\eta) = (n-1)(n-2) \Bigl[ 1- \Bigl(\fr{\eta}{A}\Bigr) ^{-{1}/{\gamma}} \Bigr]\quad
\text{ on }\eta > A.\end{equation} 

In this section, we are going to construct   sub and super solutions of equation \eqref{eq-tildew} 
in the form 
\begin{equation}\label{eqn-hw2}
 { \tw}(\eta,\tau) = \hat w_0(\eta) + e^{-2\gamma\tau}(\hat w_1(\eta)+\theta \, \tw_2(\eta))\end{equation} 
for a parameter  $\theta \in \R$.  
To this end, we will choose $\hat w_1(\eta)$ and  $\tw_2(\eta)$ to be  solutions of a  first order linear ODE with different inhomogeneous terms. By setting 
\be\label{eqn-f12}
f_1(\eta):=-(n-1) \, \fr{(\hat w_0)_{\eta\eta}}{\hat w_0}  \quad \mbox{and} \quad
f_2(\eta):=-(n-1) \, \fr{(\hat w_0)_\eta ^2}{\hat w_0 ^2},\ee
we will choose  $w_1, w_2$ to be solutions of  the equations 
\be
\begin{aligned}\label{eq-linearizationouter}
\gamma \eta\,  (\hat w_1)_\eta +{(1+2\gamma)} \,  \hat w_1 &= f_1(\eta) \qquad\text{on }\eta>A\\
\gamma \eta \, (\hat w_2)_\eta +{(1+2\gamma)} \,  \hat  w_2  &=f_2(\eta) \qquad\text{on } \eta>A.
\end{aligned}
\ee
Plugging  $\tw(\eta,\tau)$ given by \eqref{eqn-hw2} into the equation  gives that 
 \be\label{eq-Bw}
\frac{e^{2\gamma\tau}}{n-1}  B[\tw] =  \Bigl  (\fr{(\hat w_0)_{\eta\eta}}{\hat w_0}+\theta\, \fr{(\hat w_0)_\eta ^2}{\hat w_0 ^2} \Bigr) - \Bigl (\fr{\tw_{\eta\eta}}{\tw}+\fr{n-6}{4}\fr{\tw_\eta ^2}{\tw ^2} \Bigr )\ee
where  $B[\cdot]$ is given by \eqref{eqn-derror}
Thus,  from the proposed choice of $\hat w_1, \hat w_2$ to satisfy \eqref{eq-linearizationouter}, one expects that  
 $\tw$  is a subsolution  of equation \eqref{eq-tildew} if $\theta<\fr{n-6}{4}$ and a supersolution  if $\theta>\fr{n-6}{4}$, 
 for all parameters $\gamma >0$. 
The rest  of this section  is devoted to  the justification of this idea which requires a rather delicate calculation. The case of parameters $\gamma \geq 1/2$ is shown in Proposition \ref{prop-outer} below. As we will see in the proof  Proposition 
\ref{prop-outer2} below, in the case of
parameters  $\gamma < 1/2$ one needs to add correction term to $\tw$. 

\smallskip 
Recalling  the definition of  $\hat w_0$ in \eqref{eq-tildew0} and $f_1, f_2$ in \eqref{eqn-f12} we have  
\be\label{eqn-f12-2}
\begin{split}
f_1(\eta) &= +(n-1) \, \fr{\gamma+1}{\gamma^2} \, \fr{\eta^{-\fr{1}{\gamma}-2}}{A^{-\fr{1}{\gamma}}-\eta^{-\fr{1}{\gamma}}}>0\\
f_2(\eta) &= -(n-1)\, \fr{1}{\gamma^2}\frac{\eta^{-\fr{2}{\gamma}-2}}{(A^{-\fr{1}{\gamma}}-\eta^{-\fr{1}{\gamma}})^2} <0\end{split}
\ee
hence  the  explicit solutions $w_i, i=1,2$  of $\eqref{eq-linearizationouter}$ are given by 
\begin{equation} \label{eq-tildew_i}
\hat w_i(\eta)= \fr{\eta_0^{2+\fr{1}{\gamma}}\hat w_i(\eta_0)}{\eta^{2+\fr{1}{\gamma}}}+ \fr{1}{\eta^{2+\fr{1}{\gamma}}} \, \int_{\eta_0}^{\eta} \fr{f_i(x)}{\gamma}x^{1+\fr{1}{\gamma}}\, dx, \quad \eta >A.
\end{equation}
We will now fix   the functions  $\tw_1$ and $\tw_2$ by fixing their values at a given point $\eta_0 > A$.  
While doing so, we will choose $\hat w_2$ to be a positive function on $\eta >A$.   Indeed, since   $f_2(\eta)\, \eta^{1+\fr{1}{\gamma}}$  is integrable near $\eta=\infty$,  we may  choose 
$$\hat w_2(\eta):= -  \fr{1}{\eta^{2+\fr{1}{\gamma}}} \, \int_{\eta}^{+\infty}  \fr{f_2(x)}{\gamma}x^{1+\fr{1}{\gamma}}\, dx >0$$
that is we choose ${\ds \hat w_2(\eta_0):= - \gamma^{-1} \,  \eta_0^{-(2+\fr{1}{\gamma})} \int_{\eta_0}^{+\infty}  f_2(x)\, x^{1+\fr{1}{\gamma}} dx}$ in \eqref{eq-tildew_i}. For the function $\hat w_1$ we may choose any value as  $\hat w_1(\eta_0)$, since 
we do not need to choose it to be positive.  Note that by choosing $\hat w_2$ to be positive, the family of functions
$\tw_1+\theta \, \tw_2$ is monotone in $\theta \in \R$. To simplify the notation we will simply set from now on
\be\label{eqn-dh}
\hat h:= \tw_1+\theta \, \tw_2. 
\ee
We will  investigate the behavior of the  family  $ \hat h(\eta):=\big ( \tw_1+\theta \tw_2 \big )(\eta)$, $\theta \in \R$,  as  $\eta  \to A^+$ and
$\eta \to +\infty$.  We will first see that the  behavior near  $\eta =A^+$   is governed by $\tw_2$.

\begin{lemma} \label{prop-tildew1at1} For any linear combination $\hat h:=\tw_1+\theta  \tw_2$ of the solutions $\tw_1, \tw_2$
chosen above we have, we have  
\bee
\begin{split}
&\hat h (\eta) \, = +\fr{\theta}{\gamma A}\fr{n-1}{(\eta-A)}+ o\Bigl(\fr{1}{(\eta-A)}\Bigr)\\ 
&\hat h'(\eta) = -\fr{\theta}{\gamma A}\fr{n-1}{(\eta-A)^2}+o\Bigl(\fr{1}{(\eta-A)^2}\Bigr) \\
&\hat h''(\eta) = +\fr{2\theta}{\gamma A}\fr{n-1}{(\eta-A)^3}+o\Bigl(\fr{1}{(\eta-A)^3}\Bigr)
\end{split}
\eee
as $\eta \to A^+$. 
\end{lemma}
\begin{proof} 
By \eqref{eq-tildew_i} we have 
\begin{equation} \label{eq-tildew_1}
\hat h (\eta)= \fr{\eta_0^{2+\fr{1}{\gamma}}\big (\tw_1+\theta \tw_2\big ) (\eta_0)}{\eta^{2+\fr{1}{\gamma}}}+ \fr{1}{\eta^{2+\fr{1}{\gamma}}}\int_{\eta_0}^{\eta} \fr{\big (f_1+\theta f_2\big )(x)}{\gamma}x^{1+\fr{1}{\gamma}}dx.
\end{equation}
Now, by Taylor's theorem, \eqref{eqn-f12-2} and derivatives of these equations imply the following behavior as $\eta \to A^+$ 
 $$(f_1+\theta f_2)(\eta)= -\fr{\theta(n-1)}{(\gamma -A)^2}+O((\eta-A)^{-1})$$ and $$(f_1+\theta f_2)'(\eta)= 2\fr{\theta(n-1)}{(\gamma -A)^3}+O((\eta-A)^{-2}).$$ In particular we see from the above that $\hat w_2$ dominates as $\eta \to A^+$ and by  L'H\^{o}pital's rule on $(\eta-A)\int_{\eta_0}^{\eta} \eta^{-1}{\big (f_1+\theta f_2\big )(x)}x^{1+\fr{1}{\gamma}}dx$ we obtain 
 $$\lim_{\eta\to A+} (\eta-A) \hat h(\eta)=\fr{(n-1)\theta}{\gamma A}.$$
Similarly,  taking derivatives in   \eqref{eq-tildew_1} and using the asymptotics for $(f_1+\theta f_2)$ and $(f_1+\theta f_2)'$,
we obtain 
 $$\lim_{\eta\to A+}(\eta-A)^2 \hat h'(\eta)=- \fr{(n-1)\theta}{\gamma A},\quad  \lim_{\eta\to A+} (\eta-A)^3
 \hat h''(\eta)=\fr{2(n-1)\theta}{\gamma A}.$$

\end{proof}

\begin{remark} Lemma \ref{prop-tildew1at1} shows that  $\tw(\eta,\tau)$ defined by \eqref{eqn-hw2}, 
which is the ansatz for super and sub solutions,  for $\theta>\fr{n-6}{4}$ and $\theta<\fr{n-6}{4}$ respectively, predicts the correct lower order asymptotic behavior in the outer region which is then matched with the inner region where a steady soliton comes in. Indeed, in our previous 
work \cite{CD} we showed  that a  steady soliton of \eqref{eq-flatyamabe} satisfies the asymptotics
\be\label{eq-solitonasymptotics}\bar w (\xi) = \fr{(n-1)(n-2)}{\gamma A} \xi + \kappa+ \fr{(n-1)(n-6)}{4\gamma A} \fr{1}{\xi}+o\Bigl(\fr{1}{\xi}\Bigr).\ee
On the other hand, Lemma \ref{prop-tildew1at1} yields 
\bee\tw(\eta,\tau) = \fr{(n-1)(n-2)}{\gamma A} \big (\eta-A \big ) + e^{-2\gamma\tau}\, \fr{(n-1)\theta}{\gamma A}\fr{1}{\eta-A}+
O \big ((\eta-A)^2\big). \eee
Hence, if we match  the inner-outer variables by setting   $\eta= A + \xi e^{-\gamma\tau}$,  we obtain  
\bee
e^{\gamma\tau}\tw(\eta,\tau) = \fr{(n-1)(n-2)}{\gamma A} \, \xi + \fr{(n-1)\theta}{\gamma A}\, \fr{1}{\xi}+O\Bigl((\xi^2+1)e^{-\gamma\tau}\Bigr).\eee
Also, this suggests that in the  inner region  $\bar w$ should be $\bar w_0$, the translating soliton which satisfies asymptotics \eqref{eq-solitonasymptotics} with  $\kappa=0$. 
\end{remark}

We  will next see that behavior of  $\big ( \tw_1+\theta \tw_2  \big )(\eta)$ as $\eta \to \infty$  is governed by $\tw_1$.

\begin{lemma}\label{prop-tildew1infty} For any linear combination $\hat h:= \tw_1+ \theta \,  \tw_2$ of the solutions $\tw_1, \tw_2$
chosen above we have, we have  
\bee
\begin{split} &\hat h (\eta)\,   =\fr{(n-1)(1+\gamma)}{\gamma^3} A^{\fr{1}{\gamma}}\eta^{-\fr{1}{\gamma}-2}\ln \eta+o(\eta^{-\fr{1}{\gamma}-2}\ln \eta)\\
&\hat h' (\eta) =-\fr{(n-1)(1+\gamma)(1+2\gamma)}{\gamma^4}A^{\fr{1}{\gamma}}\eta^{-\fr{1}{\gamma}-3} \ln \eta+o(\eta^{-\fr{1}{\gamma}-3}\ln \eta) \\
&\hat h'' (\eta) =\fr{(n-1)(1+\gamma)(1+2\gamma)(1+3\gamma)}{\gamma^5}A^{\fr{1}{\gamma}}\eta^{-\fr{1}{\gamma}-4} \ln \eta +o(\eta^{-\fr{1}{\gamma}-4}\ln \eta)
\end{split}
\eee
as $\eta \to \infty$.
\end{lemma}

\begin{proof}
The proof is similar as in Lemma  \ref{prop-tildew1at1} if we check the asymptotics 
\bee
\begin{split}\label{eq-fexpansioninfty}
f_1(\eta)&=\fr{(n-1)(1+\gamma)}{\gamma^2}\, \big (A^{\fr{1}{\gamma}}\eta^{-\fr{1}{\gamma}-2}+A^{\fr{2}{\gamma}}x^{-\fr{2}{\gamma}-2}\big )+ o(x^{-\fr{2}{\gamma}-2})\\
f_2(\eta)&=-\fr{(n-1)}{\gamma^2}\big (A^{\fr{2}{\gamma}} \eta^{-\fr{2}{\gamma}-2}+2A^{\fr{3}{\gamma}}x^{-\fr{3}{\gamma}-2}\big )+ o(\eta^{-\fr{3}{\gamma}-2})
\end{split}
\eee and corresponding asymptotics for $f_1' $ and $f_2'$
as $ \eta \to\infty$.  

\end{proof}

In the next lemma we give   more  precise asymptotics   which will be used later when we have to add higher order terms 
to barrier in the case $0< \gamma \le \fr{1}{2}$.  Notice that $\eta^{- \frac 1\gamma -2}$ is a solution of the homogenous equation of \eqref{eq-linearizationouter} and we have  constants $C$  in the lemma below since we haven't chosen a specific $\tw_1$.

\begin{lemma}\label{remark-preciseasymp}
For any linear combination $\hat h:= \tw_1+\theta  \tw_2$ of the solutions $\tw_1, \tw_2$
chosen above, we have  
\bee
 \begin{split} \label{eq-11} &\hat h(\eta) =+\fr{(n-1)(1+\gamma)}{\gamma^3} A^{\fr{1}{\gamma}}\eta^{-\fr{1}{\gamma}-2}\ln \eta +C\eta^{-\fr{1}{\gamma}-2} \\
&\qquad  \quad -(n-1)\fr{(n-1)(1+\theta+\gamma)}{\gamma^2}A^{\fr{2}{\gamma}} \gamma ^{-\fr{2}{\gamma}-2}+o(\eta^{-\fr{2}{\gamma}-2})\\
&\hat h' (\eta) =-\fr{(n-1)(1+\gamma)(1+2\gamma)}{\gamma^4}A^{\fr{1}{\gamma}}\eta^{-\fr{1}{\gamma}-3} \ln \eta\\ &
\qquad \quad +C' \eta^{-\fr{1}{\gamma}-3}+o(\eta^{-\fr{2}{\gamma}-2})\\
&\hat h'' (\eta) =+\fr{(n-1)(1+\gamma)(1+2\gamma)(1+3\gamma)}{\gamma^5}A^{\fr{1}{\gamma}}\eta^{-\fr{1}{\gamma}-4} \ln \eta\\ & \qquad \quad +C'' \eta^{-\fr{1}{\gamma}-4}+o(\eta^{-\fr{2}{\gamma}-2})
\end{split}
\eee
as $\eta \to +\infty$. Here, $C$, $C'$, and $C''$ are constants depending on the choice of $\tw_1$.
 \end{lemma} 

\begin{proof} Can be shown in the same manner as in Lemma \ref{prop-tildew1infty}. 

\end{proof}
We will now show that $\hat w(\eta,\tau)$ given by \eqref{eqn-hw2}  is a sub or  super - solution  of equation 
\eqref{eq-tildew} in the appropriate regions.   We will first deal  with the case of parameters $\gamma > 1/2$. The case
$\gamma \leq 1/2$ is more delicate and will be considered later.

\begin{prop}\label{prop-outer} For any $\gamma>1/2$ and given $\theta \neq \fr{n-6}{4}$, there exist $\tau_0\in \R$ and $\xi_0>0$ depending on $n$, $A$, $\gamma$ and $\theta$ such that  the function 
$${ \tw}(\eta,\tau) := \hat w_0(\eta) + e^{-2\gamma\tau}(\hat w_1(\eta)+\theta \tw_2(\eta))$$ is a subsolution 
of equation  \eqref{eq-tildew}  if $\theta < \fr{n-6}{4}$ or a supersolution if $\theta > \fr{n-6}{4}$, respectively, 
in the region  $$\{ (\eta,\tau) \, | \, \eta \ge A + \xi_0 \, e^{-\gamma \tau}, \,\, \tau\ge\tau_0 \}.$$
\end{prop}

\begin{proof}[Proof of Proposition \ref{prop-outer}]
As before,  let us denote by  $$\hat h:=\tw_1+\theta \tw_2.$$
We need to show that 
\be\label{eqn-bw10}
B[\tw]<0, \,\,\, \text{ if }\theta<\fr{n-6}{4}\quad\text{ or }\quad B[\tw]>0,\,\,\, \text{ if }\theta>\fr{n-6}{4}
\ee
holds in the region 

Proposition  \ref{prop-outer} follows from the the two claims  below 

\begin{claim}\label{lem-outer1}
For any $\gamma>0$ and given ${ \theta\neq \fr{n-6}{4}}$, there exist  $\xi_0 >0$ and $\delta>0$ such that $ { \bw}(\eta,\tau)$ is a subsolution of equation 
\eqref{eq-tildew}  if ${ \theta<\fr{n-6}{4}}$   or  a supersolution    if  ${ \theta>\fr{n-6}{4}}$, respectively,  in  
the  region $$\{(\eta,\tau) \,  |\,   A+\xi_0\,  e^{-\gamma\tau} \leq \eta< A+\delta, \, \tau \in \R \, \}.$$
\end{claim} 

\begin{proof}[Proof of Claim \ref{lem-outer1}] 
By  Lemma  \ref{prop-tildew1at1}, we may find  $\kappa=\kappa(n,A, \gamma)>0$ such that 
\begin{equation}|\hat h|  (\eta-A),\ |\hat h'|  (\eta-A)^2,\ |\hat h''|  (\eta-A)^3 < \kappa \, |\theta|\end{equation}
holds on the region $A < \eta < A+1$. Moreover,  by  Taylor's theorem we may choose the constant $\kappa$ so that 
\bee
\Bigl |\hat w_0 (\eta)-\fr{(n-1)(n-2)}{\gamma A}(\eta-A)\Bigr|<\kappa\, (\eta-A)^2 
\eee and
\bee \Bigl|\tw_0'(\eta)-\fr{(n-1)(n-2)}{\gamma A}\Bigr|<\kappa\, (\eta-A) \quad \mbox{and} \quad 
\Bigl|\tw_0''(\eta)\Bigr|<\kappa \eee
hold. Using these, we get
\bee
\begin{split}\label{eq-lem1}
\Bigl| \tw-\fr{(n-1)(n-2)}{\gamma A}(\eta-A)\Bigr| &=\Bigl |\tw_0-\fr{(n-1)(n-2)}{\gamma A}(\eta-A)+e^{-2\gamma\tau}\hat h\Bigr |\\
&\le \kappa \, (\eta-A)^2+\kappa \, |\theta|\, e^{-2\gamma\tau}\fr{1}{(\eta-A)}\\&=\Bigl (\kappa(\eta-A)+\fr{\kappa\, |\theta|}{((\eta-A)e^{\gamma\tau})^2}\Bigr )\, (\eta-A)
\\&\le \kappa\, \Bigl (\delta+\fr{|\theta|}{\xi_0^2}\Bigr ) \, (\eta-A). 
\end{split}
\eee
Hence by restricting to $0<\delta<1$ small and $\xi_0>0$ large, we may assume that $$\tw> \fr{(n-1)(n-2)}{2\gamma A}\, (\eta-A)>0.$$
Using  the above we find that in the considered region we have 
\bee
\begin{split}\label{eq-lem2}
\Bigl|\fr{\tw_{\eta\eta}}{\tw}\Bigr|\le\fr{|\tw_0''|+|e^{-2\gamma\tau} \hat h''|}{\tw}&\le \fr{1}{(\eta-A)}\fr{2\gamma A}{(n-1)(n-2)}\big (\kappa+\fr{|\theta|\kappa}{(\eta-A)^3e^{2\gamma\tau}}\big )\\&\le\fr{\kappa}{(\eta-A)^2}\fr{2\gamma A}{(n-1)(n-2)}
\Bigl(\delta + \fr{|\theta|}{\xi_0^2}\Bigr)\end{split}\eee
and
\bee
\begin{split}\label{eq-lem3}
\Bigl|\fr{\tw_{\eta}^2}{\tw^2}\Bigr|&=\fr{|\tw_0'|^2|1+e^{-2\gamma\tau} \hat h'/\tw_0'|^2}{\tw^2}\\&\le\fr{1}{(\eta-A)^2}\Bigl|  \fr{\tw_0'^2}  {\tw^2\,  (\eta-A)^{-2}}\Bigl(1+\fr{|\theta|\kappa}{(\eta-A)^2e^{2\gamma\tau}\tw_0'}\Bigr)^2 \Bigr|\\
&\le\fr{1}{(\eta-A)^2}\Bigl( \fr{\fr{(n-1)(n-2)}{\gamma A} +\kappa \delta}{\fr{(n-1)(n-2)}{\gamma A}-\kappa\big (\delta+\fr{|\theta|}{\xi_0^2}\big )}\Bigr)^2\Bigl(1+\fr{|\theta|\kappa}{\xi_0^2\big (\fr{(n-1)(n-2)}{\gamma A}-\kappa\delta\big )}\Bigr )^2.\end{split}
\eee
Similarly,  we estimate from below \bee
\label{eq-lem4}
\Bigl |\fr{\tw_{\eta}^2}{\tw^2}\Bigr|\ge\fr{1}{(\eta-A)^2}\Bigl( \fr{\fr{(n-1)(n-2)}{\gamma A}-\kappa\delta}{\fr{(n-1)(n-2)}{\gamma A}+\kappa\big (\delta+\fr{|\theta|}{\xi_0^2}\big )}\Bigr)\Bigl(1-\fr{|\theta|\kappa}{\xi_0^2\, \big (\fr{(n-1)(n-2)}{\gamma 
A}-\kappa\delta\big)}\Bigr  )^2.
\eee
Thus,  for a fixed $\e>0$ to be determined later,  we may find small $\delta >0$ and large $\xi_0>0$ so that all 
\bee\begin{aligned}
\Bigl| \fr{\tw_{\eta}^2}{\tw^2}-\fr{1}{(\eta-A)^2}\Bigr|, \,\, \Bigl| \fr{\tw_0'^2}{\tw_0^2}-\fr{1}{(\eta-A)^2}\Bigr|, \,\, \Bigl| \fr{\tw_{\eta\eta}}{\tw}\Bigr|,\,\, \Bigl| \fr{\tw_0''}{\tw_0}\Bigr|\le \fr{\e}{(\eta-A)^2}.
\end{aligned} 
\eee
Recalling  the formula  for $B[\tw]$ in  \eqref{eq-Bw} and using  the  triangle inequality  successively  we  obtain 
\bee
\Bigl|\fr{e^{2\gamma\tau}B[\tw]}{(n-1)}-\Bigl(\theta-\fr{n-6}{4}\Bigr)\fr{1}{(\eta-A)^2}\Bigr|\le \e
\fr{|\theta|+\Bigl|\fr{n-6}{4}\Bigr|+2}{(\eta-A)^2}.\eee
Finally, by   choosing  ${\ds \e := \fr{1}{2}\big  |\theta-\fr{n-6}{4}\big | \big  ( |\theta|+\big |\fr{n-6}{4}\big |+2\big  )^{-1}}$ we 
 conclude that \eqref{lem-outer1}  holds,  finishing the proof of the claim. 

\end{proof}

We will next proceed in our next claim which holds for  $\gamma> 1/2$. 

\begin{claim}\label{lem-outer2}
For $\gamma>1/2$ and given $\theta\neq\fr{n-6}{4}$ and $\delta>0$, there is $\tau_0=\tau_0(\theta,\delta,\gamma)$ such that such that $ \hat w(\eta,\tau)$ is a subsolution of equation 
\eqref{eq-tildew}  if $\theta<\fr{n-6}{4}$   or  a supersolution    if  $\theta>\fr{n-6}{4}$ on  the
set $$\{(\eta,\tau) \,  | \, \eta > A+\delta,  \,  \, \tau \geq \tau_0\}.$$
\end{claim}

\begin{proof}[Proof of Claim \ref{lem-outer2}]  For a given $\delta_0>0$,
Lemma  \ref{prop-tildew1infty} and the asymptotic behavior of $\hat w_0$, $\hat w_0'$  and $\hat w_0''$ can be used to find a constant $\kappa>0$ such that 
\begin{equation*}\label{eq-lem2-1}
\Bigl|\fr{\hat h}{w_0}\Bigr|,\, \Bigl|\fr{\hat h'}{w_0'}\Bigr|, \, \Bigl|\fr{\hat h''}{w_0''}\Bigr| < \kappa \, \fr{\ln( 1+\eta)}{\eta ^2}< \fr{\kappa}{A+\delta_0}\qquad\text{on } \,\, \eta>A+\delta_0.
\end{equation*}
Thus, we may start with some large $\tau_0$ such that $$\text{$\tw$, $\tw_\eta$, $\tw_{\eta\eta}$ have same the sign as  $\hat w_0$, $\tw_0'$, $\tw_0''$ respectively}$$ on $\eta > A+ \delta$ and $\tau>\tau_0$. In particular, they are nonzero. Using 
the formula for $B[\tw]$ in \eqref{eq-Bw}, we write  
 \bee\fr{e^{2\gamma\tau}}{(n-1)}\, B[\tw]=\Bigl( \fr{\tw_0''}{\tw_0}-\fr{\tw_{\eta\eta}}{\tw}\Bigr)+ \fr{n-6}{4}\Bigl(\fr{\tw_0'^2}{\tw^2}-\fr{\tw_{\eta}^2}{\tw^2}\Bigr)+\big (\theta-\fr{n-6}{4}\big ) \fr{\tw_0'^2}{\tw_0^2}. \eee
We will  show that the first two terms become arbitrarily small in comparison  with the last term,  for  $\tau_0 \gg 1$ 
(and hence $\tau \gg 1$). Indeed, we have 
\bee
\Bigl| \fr{\tw_{\eta\eta}}{\tw}- \fr{\tw_0''}{\tw_0}\Bigr|=e^{-2\gamma\tau}\Bigl|\fr{\tw_0''}{\tw_0}\Bigr| \, \Bigl | \fr{(\hat h''/\tw_0'')-(\hat h/\tw_0)}{1+e^{-2\gamma\tau}(\hat h/\tw_0)}\Bigr| 
\eee
and
\bee
\Bigl| \fr{\tw^2_{\eta}}{\tw^2}- \fr{\tw_0'^2}{\tw_0^2}\Bigr|=e^{-2\gamma\tau} \, \fr{\tw_0'^2}{\tw_0^2}\,  \fr{ \big | 2((\hat h'/\tw_0')-(\hat h/\tw_0))+e^{-2\gamma\tau}((\hat h'/\tw_0')^2-(\hat h/\tw_0)^2) \big |}{(1+e^{-2\gamma\tau}(\hat h/\tw_0))^2}.
\eee
Thus,   the asymptotics  in Lemma \ref{remark-preciseasymp} imply that   by choosing $\tau_0 \gg 1$ we have 
\begin{equation}
\begin{split}
&\Bigl | \fr{\tw_{\eta\eta}}{\tw}- \fr{\tw_0''}{\tw_0}\Bigr| \le 10 \, e^{-2\gamma\tau} \, \kappa\, \Bigl | \fr{\tw_0''}{\tw_0}\Bigr| \, \fr{\ln (1+\eta)}{\eta^2}\\
&\Bigl| \fr{\tw^2_{\eta}}{\tw^2}- \fr{\tw_0'^2}{\tw_0^2}\Bigr| \le 10 \,   e^{-2\gamma\tau}\, \kappa \, \Bigl ( \fr{\tw_0'^2}{\tw_0^2} 
\Bigr )\,  \fr{\ln (1+\eta)}{\eta^2}.
\end{split}
\end{equation}
for all $\tau > \tau_0$. In the case  $\gamma >\fr{1}{2}$, by $\fr{\ln\eta}{\eta^2}$, $e^{-2\gamma\tau}$ terms in the previous estimate and asymptotics \eqref{eq-fexpansioninfty}, we can make $\tau_0$ large and conclude
\begin{equation}
\Bigl| \Bigl(\fr{\tw_0''}{\tw_0}-\fr{\tw_{\eta\eta}}{\tw}\Bigr)+ \fr{n-6}{4}\Bigl(\fr{\tw_0'^2}{\tw^2}-\fr{\tw_{\eta}^2}{\tw^2}\Bigr)\Bigr|
\le \fr{1}{2} \, \big | \theta -\fr{n-6}{4} \big |\fr{\tw_0'^2}{\tw_0^2}
\end{equation}
on the considered region.  This proves that for $\tau \geq \tau_0 \gg1$, \eqref{eqn-bw10} holds, 
finishing the proof of the claim. 
\end{proof}

To finish the proof of the proposition, for any given $\gamma > 1/2$ and   $\theta \neq \frac{n-6}4$, 
Claim \ref{lem-outer1} implies that there exists  $\delta >0$ such that 
such that  \eqref{eqn-bw10} holds in the region $A+\xi_0\,  e^{-\gamma\tau} \leq \eta< A+\delta$. In addition,  by  Claim \ref{lem-outer2},
there exists   $\tau_0=\tau_0(\theta,\delta,\gamma)$ such that \eqref{eqn-bw10} holds in the region $\eta > A+\delta,  \,   \tau \geq \tau_0$.
We conclude that \eqref{eqn-bw10} holds in the whole outer region $\eta > A+\xi_0\,  e^{-\gamma\tau}$ for $\tau \geq \tau_0$ finishing the proof
of the proposition.

\end{proof}

In the case  $0< \gamma \le 1/2$, we need to add a  higher order correction term  in our barrier.  
For integers  $k\ge 2$ and $0\le l \le k$, we  define the functions 
\begin{equation}\label{eqn-vkl} 
 v_{k,l} (\eta) :=  \eta ^{-2k-\fr{1}{\gamma}} \, (\ln \eta)^l, \qquad \eta >1.\end{equation}
They satisfy  the following relation 
\be\label{eqn-relation}
(1+2k\gamma) \, v_{k,l} + \gamma \eta \, v'_{k,l} =\begin{cases}\begin{aligned} &\gamma l  \,  v_{k,l-1}\quad &&\text{ if }l>0\\&0\quad &&\text{ if }l=0\end{aligned}\end{cases}
\ee
and to simplify the notation we also set  $v_{k,-1}(\eta)=0$ and  $v_{k,-2}(\eta)=0$. We will show the following. 

\begin{prop}\label{prop-outer2} For any  $0<\gamma \le 1/2$ and  given $\theta \neq \frac{n-6}4$,  there exist $\tau_0\in \R$, $\xi_0>0$, integer 
$N \geq 2$ and coefficients $\{ c_{k,l} \}_{ 2 \leq k \leq N, \,  0 \le l \le k}$  such that  the function 
$$\tw (\eta,\tau):= \hat w_0(\eta) +e^{-2\gamma\tau}(\hat w_1(\eta)+\theta\tw_2 (\eta))+ \Sigma _{k=2}^{N}   e^{-2k\gamma\tau}  \Sigma_{l=0}^k c_{k,l}  \,  v_{k,l} (\eta) $$ is a subsolution 
of equation  \eqref{eq-tildew}  if $\theta < \fr{n-6}{4}$ or a supersolution if $\theta > \fr{n-6}{4}$, in the region  $$\{ (\eta,\tau) \, | \, \eta \ge A + \xi_0 \, e^{-\gamma \tau}, \,\, \tau\ge\tau_0 \}.$$\end{prop}

\begin{proof}[Proof of Proposition \ref{prop-outer2}]
For the given $0< \gamma \le 1/2$, let $N$ denote the smallest integer making $\gamma > 1/(2N)$, namely $N:=[1/(2\gamma)]+1$. 
The next claim  corresponds to Claim  \ref{lem-outer2} for this case.

\begin{claim}\label{claim-outer3}
For any $0< \gamma \le 1/2$ and   any given choice of $\{c_{k,0} \}_{2 \le k \le N} $, there are coefficients $\{c_{k,l} \}_{2  \le k \le N, 0 \le l \le k}$ 
so that for any given $\delta >0$ and $\theta \neq \fr{n-6}{4}$,  the function 
 \bee \label{eq-higherform}\tw (\eta,\tau)= \hat w_0(\eta) +e^{-2\gamma\tau}(\hat w_1(\eta)+\theta\, \tw_2 (\eta))+\Sigma _{k=2}^{N}   e^{-2k\gamma \tau}  \Sigma_{l=0}^k c_{k,l}  \,  v_{k,l} (\eta) \eee  is  a subsolution 
of equation 
\eqref{eq-tildew}  if $\theta<\fr{n-6}{4}$   or  a supersolution    if  $\theta>\fr{n-6}{4}$ on  the
set  $$\{(\eta,\tau) \ |\ \eta > A+\delta, \,\, \tau \ge \tau_0\}$$
where  $\tau_0=\tau_0(\gamma, \theta, \delta) \gg 1$.
\end{claim}

\begin{proof}[Proof of Claim \ref{claim-outer3}]
Let us assume $\theta < \fr{n-6}{4}$ because the other case follows similarly.  Suppose $\tw$ is of the  form of \eqref{eq-higherform}. We will choose the coefficients  $c_{k,l} $ later. 

We split the operator $B[\cdot]$ given by \eqref{eqn-derror} into linear and nonlinear parts, that is  we write  
\begin{equation} B[\hat w]= I_1[\tw] +(n-1)(n-2) - I_2[\tw].
\end{equation}
 where
 $$I_1[\tw] := \partial_\tau \hat w -\gamma\eta \, \hat w_\eta -\hat w, \quad I_2[\tw]:= (n-1)\Bigl(\fr{\hat w_{\eta\eta}}{\hat w}+\fr{n-6}{4}\fr{\hat w_\eta ^2}{\hat w ^2} \Bigr)e^{-2\gamma \tau}.$$
Then, using \eqref{eqn-f12}, \eqref{eq-linearizationouter} and \eqref{eqn-relation} we find 
 \bee \begin{aligned} I_1[\tw] = &I_1[\tw _0] + I_1[ \hat h \, e^{-2\gamma\tau} ]+ \Sigma_{k=2}^{N}\Sigma_{l=0}^k c_{k,l}\, I_1[ v_{k,l} \, e^{-2k\gamma\tau}] \\ 
 =& -(n-1)(n-2)+  (n-1)\Bigl(\fr{(\hat w_0)_{\eta\eta}}{\hat w_0}+\theta\fr{(\hat w_0)_\eta ^2}{\hat w_0 ^2} \Bigr)e^{-2\gamma\tau} \\
 &- \Sigma_{k=2}^{N}  e^{-2k\gamma\tau}\Sigma_{l=1}^k \gamma l \, c_{k,l}   \, v_{k,l-1} .\end{aligned}\eee
Meanwhile,  using the asymptotics of $\tw_0$, $\tw_1+\theta\tw_2$, $v_{k,l} $ and their derivatives
\bee\begin{split} 
I_2[\tw]  &= (n-1)\Bigl (  \fr{\hat w_{\eta\eta}}{\hat w} + \fr{n-6}{4} \fr{(\hat w_0)_\eta ^2}{\tw_0^2}+ o(\eta^{-2-\fr{2}{\gamma}})e^{-2\gamma \tau} \Bigr ) e^{-2\gamma \tau}\\
 =& (n-1)\Bigl (  \fr{\hat w_{\eta\eta}}{\hat w_0} + \fr{n-6}{4} \fr{(\hat w_0)_\eta ^2}{\tw_0^2}+o(\eta^{-2-\fr{2}{\gamma}})e^{-2\gamma \tau} \Bigr ) e^{-2\gamma \tau}\\
 =& (n-1)\Bigl (   \fr{ (\tw_0)_{\eta\eta}}{\hat w_0} + \fr{n-6}{4} \fr{(\hat w_0)_\eta ^2}{\tw_0^2} \Bigr )  e^{-2\gamma \tau} \\&+ (n-1) \,  \fr{ \hat h_{\eta\eta} e^{-4\gamma\tau} + \Sigma_{k=2}^N  e^{-2(k+1) \gamma \tau} \, \Sigma_{l=0}^k c_{k,l} (v_{k,l} )_{\eta\eta}}{\tw_0} + o(\eta^{-2-\fr{2}{\gamma}})e^{-4\gamma \tau} \\
 =& (n-1)\Bigl (   \fr{ (\tw_0)_{\eta\eta}}{\hat w_0} + \fr{n-6}{4} \fr{(\hat w_0)_\eta ^2}{\tw_0^2} \Bigr )  e^{-2\gamma \tau} \\&+ \fr{ \hat h_{\eta\eta} e^{-4\gamma\tau} + \Sigma_{k=2}^{N-1}  e^{-2(k+1) \gamma \tau} \, \Sigma_{l=0}^k  c_{k,l} (v_{k,l} )_{\eta\eta}}{(n-2)}  + o(\eta^{-2-\fr{2}{\gamma}})\, e^{-4\gamma \tau}
\end{split}.\eee
In the last line we used that $2N > 1/\gamma$.  Also, 
 $g(\eta,\tau)=o(\eta^{-2-\fr{2}{\gamma}})$ means that $\sup_{\eta>\eta', \tau>\tau'} 
 \eta^{2+\fr{2}{\gamma}}  g(\eta,\tau) \to 0 $,  as $\eta'\to\infty$ for any fixed $\tau'$. 

Combining the above computations yields 
\bee \begin{split}
B[\tw] =& (n-1)\Bigl(\theta - \fr{n-6}{4}\Bigr)\fr{(\hat w_0)_\eta ^2}{\tw_0^2}e^{-2\gamma\tau} - \Sigma_{k=2}^{N} \, e^{-2k\gamma\tau}\, 
\Sigma_{l=1}^k \gamma l \, c_{k,l}   \, v_{k,l-1}  \\
& -\fr{ \hat h_{\eta\eta} e^{-4\gamma\tau} + \Sigma_{k=2}^{N-1}  \, e^{-2(k+1) \gamma \tau}\, \Sigma_{l=0}^k  c_{k,l} \, (v_{k,l} )_{\eta\eta}}{(n-2)}  + o(\eta^{-2-\fr{2}{\gamma}}) e^{-4\gamma \tau}\\
=& (n-1)\Bigl(\theta - \fr{n-6}{4}\Bigr)\fr{(\hat w_0)_\eta ^2}{\tw_0^2}e^{-2\gamma\tau} -\Bigl (\fr{h_{\eta\eta}}{n-2} + \Sigma_{l=1}^2 \gamma l \, c_{2,l} v_{2,l-1} \Bigr )  e^{-4\gamma\tau} \\
&-\Sigma_{k=3}^N\Bigl ( \fr{1}{n-2}\Sigma_{l=1}^{k} c_{k-1,l-1}( v_{k-1,l-1})_{\eta\eta} +\Sigma_{l=1}^k \gamma l \, c_{k,l}   v_{k,l-1} \Bigr ) \, e^{-2k\gamma\tau} \\&+o(\eta^{-2-\fr{2}{\gamma}}) e^{-4\gamma \tau}.
\end{split}
\eee
Let us remark   that $(v_{k-1,l-1})_{\eta\eta} $ can  be written as a linear combination of $\{  v_{k,l-1},v_{k,l-2},v_{k,l-3}\}$.
Hence for any given  $\{c_{k,0}\} _{2\le k\le N}$,  there is a  unique choice $\{c_{k,j}\}_{2\le k \le N, \ 1\le l \le k}$ such that \bee B[\tw] = (n-1)\Bigl(\theta - \fr{n-6}{4}\Bigr)\fr{(\hat w_0)_\eta ^2}{\tw_0^2}e^{-2\gamma\tau} +o(\eta^{-2-\fr{2}{\gamma}})e^{-4\gamma \tau}. \eee Here we also used the asymptotic expansion of $h_{\eta\eta}$ as $\eta \to \infty$, namely 
\bee
h_{\eta\eta} = (n-1)\, A^{\fr{1}{\gamma}}   \fr{(1+\gamma) (1+2\gamma)(1+3\gamma)}{\gamma^5}  v_{2,1}+C'' v_{2,0}+ o(\eta^{-2-\fr{2}{\gamma}})
\eee which has been shown in Lemma  \ref{remark-preciseasymp}.
Finally, we may find a large $\tau_0 =\tau_0(\delta, \gamma, \theta)$ such that $B[\tw] <0$ on the region $\eta >  A+\delta$
for $\tau \geq \tau_0.$  This finishes the proof of our claim. 

\end{proof}

As we fixed $\tw_1$ and $\tw_2$ in the proof of Proposition \ref{prop-outer}, from now on let us fix $c_{k,l} $
so that Claim \ref{claim-outer3} holds,   by choosing $c_{k,0}=0$.
Next, we give  the  analogue of Claim  \ref{lem-outer1} in this case.
\begin{claim}\label{claim-outer4} 
For given $0<\gamma\le 1/2$, $\theta\neq\fr{n-6}{4}$, there exist $\xi_0>0$ and $\delta>0$ such that $$\tw (\eta,\tau)= \hat w_0(\eta) +e^{-2\gamma\tau}(\hat w_1(\eta)+\theta\tw_2 (\eta))+\Sigma _{k=2}^{N}  e^{-2k\gamma \tau}\, \Sigma_{l=0}^k c_{k,l} \,  v_{k,l} (\eta) $$ is 
is  a subsolution  of equation  \eqref{eq-tildew}  if $\theta<\fr{n-6}{4}$   and  a supersolution    if  $\theta>\fr{n-6}{4}$ on  the
region  $$\{(\eta,\tau)  \  |\  A+\xi_0 e^{-\gamma\tau}<\eta<A+\delta, \ \tau > 0\}.$$
\end{claim} 
\begin{proof}[Proof of Claim \ref{claim-outer4}]
By rewriting $\tw(\eta,\tau) = \tw_0(\eta) + e^{-2\gamma\tau} \, \hat h(\eta,\tau)$, we have the same estimate of Proposition \ref{prop-tildew1at1} and the proof is actually the same as of Claim  \ref{lem-outer1}.
\end{proof}

The proof of the Proposition \ref{prop-outer2} now  readily follows by combining claims \ref{claim-outer3} and \ref{claim-outer4}.
Let us fix  $0<\gamma\le 1/2$ and $\theta\neq \fr{n-6}{4}$. Let   $c_{k,l}$ be coefficients  
with $c_{k,0}=0$ be so that Claim \ref{claim-outer3} holds. For that choice of $c_{k,l}$, Claim \ref{claim-outer4} gives the existence
of $\xi_0 >0$ and $\delta >0$ so that  $\hat w$ is a subsolution (supersolution)  in the region  
$A+\xi_0 e^{-\gamma\tau}<\eta<A+\delta, \ \tau > 0$. By Claim \ref{claim-outer3} there exists $\tau_0=\tau_0(\gamma,\theta,\delta)$ such that $\hat w$ is a subsolution (supersolution)  in the region $\eta > A+\delta, \, \tau \geq \tau_0$. We conclude that 
$\hat w$ is a subsolution (supersolution)  in the region $\eta > A+\xi_0 e^{-\gamma\tau}, \, \tau \geq \tau_0$. Since $\delta=\delta(\gamma,\theta)$ we also have that $\tau_0=\tau_0(\gamma,\theta)$.

\end{proof} 

\section{Barrier construction in the inner region}\label{sec-inner}
We will now construct the appropriate barrier in the {\em inner region} which is the region where  
$$e^{\gamma \tau} \, \hat w(\eta,\tau) = O(1), \qquad \mbox{as} \,\,  \tau \to +\infty.$$
In this region we define  $\bar w(\xi,\tau)$ as in  \eqref{eqn-cv2},  that is we set $\bar w(\xi,\tau) = e^{\gamma \tau} \, \hat w(\eta,\tau)$,
$\xi=(\eta - A)\, e^{\gamma \tau}$. We have seen in section \ref{sec-formal} that   $\bar w (\xi,\tau)$ 
satisfies the equation $I[\bar w]=0$ with $I[\cdot]$ given by \eqref{eq-barw}. Let us assume that  in this region  
the first term in \eqref{eq-barw} having $e^{-\gamma\tau}$ becomes  negligible  as $\tau\to\infty$. Then, 
we expect that the solution $\bar w_0(\xi)$ of equation 
\bee \fr{(\bar w_0)_{\xi\xi}}{\bar w}+\fr{n-6}{4}\fr{(\bar w_0)_\xi ^2}{\bar w ^2} -(n-1)(n-2) + \gamma A \,  (\bar w_0)_\xi  =0
\eee
is the leading order term for $\bar w(\xi,\tau)$ in this region. 
We are going to find super and sub solutions $ \bw^+$ and $\bw^-$, respectively in the following form
\begin{align*}
&\bw^+ (\xi,\tau)= \fr{1}{1+\e} \bw_0 (\xi+C_1(\tau))\\
&\bw^-(\xi,\tau)=\fr{1}{1-\e} \bw_0(\xi+ C_2(\tau)).
\end{align*}
Here, $\e>0$ is a small constant and $C_1(\tau)$ and $C_2(\tau)$ are smooth functions of $\tau$. Both  will be chosen later and will depend on  $\xi_0$ which appears  in the construction of our  barriers in the outer region. 
As we will see below, the construction   is rather straightforward. 

If we plug these into $I[\cdot]$, we get 
\begin{align*}
&I[\bw^+] =+ \e \bw^+_\xi +e^{-\gamma\tau}(C_1'(\tau)\, \bw^+_\xi - (1+\gamma) \, \bw^+) \\ 
&I[\bw^-] = -\e \bw^-_\xi +e^{-\gamma\tau}(C_2'(\tau)\, \bw^+_\xi - (1+\gamma) \, \bw^-).
\end{align*}

We will next show the  following. 

\begin{prop}\label{prop-inner} Let  $0<\e<1$ and $\tau_0\in \R$. If $|C_1'(\tau)|,|C_2'(\tau)| \leq M$ on $\tau \geq\tau_0$, then 
$\xi_1$  there exist $\tau_1=\tau_1(\e,M,\xi_1)\ge\tau_0$ such that $\bw^+$ or $\bw^-$ are super or  sub solutions of equation 
\eqref{eqn-bw}  respectively, in  the region $(\xi,\tau)\in (-\infty,\xi_1)\times (\tau_1,\infty)$.

\begin{proof} For any two functions $f(s), g(s)$ we use the notation
$$f(s)\sim g(s), \,\,   \mbox{as} \,\, s\to\infty \qquad \mbox{iff} \qquad     c<\Bigl|\fr{f(s)}{g(s)}\Bigr|<C, \,   \mbox{for} \,\, s \gg 1$$
for some fixed constants $c >0$, $C < +\infty$.  

In this proof we will use the asymptotics for  $\bar w_0(s)$ and $\bar w_0 '(s)$,   as $s\to \infty$, 
which were   shown in Proposition 2.1 in \cite{CD} or \cite{DS, H}.  
Since $\bar w_0(s)=e^{2s}\, \bar U^{1-m}(e^s)$ and $\bar U^{1-m}(|x|) \, \delta_{ij} $ is a smooth radial metric on $\R^n$, we have
$$\bw_0\sim (\bw_0)_s \sim e^{2s}, \qquad  \mbox{as} \,\, s\to\infty.$$ Moreover since 
$\bw_0 \sim s, \, (\bw_0)_s\sim 1$, as $s\to\infty$, it is clear  that there is some $\tau_2$ and $c$ so that 
\bee e^{-\gamma\tau} \, (1+\gamma)\, \bw_0(s)<\fr{\e}{2}(\bw_0)_s(s), \quad 
(s,\tau)\in(-\infty,c\, e^{\gamma\tau})\times(\tau_2,\infty).\eee
Now given $\xi_1$ and   $C_1(\tau)$ and $C_2(\tau)$ satisfying the  conditions in our proposition,    we can find some $\tau_1>\max(\tau_0,\tau_2)$ such that \begin{equation}\label{eq-lempf2}\xi_1+C_i(\tau)<c\, e^{\gamma\tau} \ \,  \text{ and } \ \,
\quad |e^{-\gamma\tau}C'_i(\tau)|<\fr{\e}{2}, \quad \text{ for } \tau>\tau_1.\end{equation} 
The last two formulas  and the  fact that $w_0>0$, $(w_0)_s>0$,  imply that $I[\bw^+]>0$ and $I[\bw^-]<0$ on $(\xi,\tau)\in$ $(-\infty,\xi_1)\times (\tau_1,\infty)$, as claimed.\end{proof}\end{prop}

\section{Construction of super  and sub-solutions}\label{sec-barriers} 

In this section we will combine the results from Sections \ref{sec-outer} and \ref{sec-inner} to construct a family of super-solutions $w^+_\e$ and  sub-solutions $w^-_\e$ of  equation \eqref{eq-w} which is
equivalent to the conformally flat Yamabe flow \eqref{eq-flatyamabe} under rotational symmetry and after the cylindrical change of variables \eqref{eq-uw}. This will give a family of rotationally symmetric super and sub solutions of equation \eqref{eq-flatyamabe}  which we will then be used  in the next section to analyze the type II blow up behavior of any  solution $u(x,t)$ of \eqref{eq-flatyamabe} which satisfies the assumptions of Theorem \ref{thm-main1}.

We begin by fixing  $$\gamma >0, \qquad A > 0,  \qquad  \theta^+>\fr{n-6}{4}, \qquad \theta^-<\fr{n-6}{4}.$$
For these choices of parameters and following the results in Section \ref{sec-outer}, we define the super and sub-solutions $\tw^+$ and $\tw^-$ {\em corresponding to $\theta^+$ and $\theta^-$ respectively}  in the {\em outer region} $\eta > A+\xi_0\, e^{-\gamma\tau}$ 
 separately for different ranges of $\gamma$: for 
 $\gamma>1/2$ we set 
\bee\tw^{\pm}(\eta,\tau):=\hat w_0(\eta) + e^{-2\gamma\tau} \, \big ( \hat w_1(\eta)+\theta^{\pm} \, \hat w_2(\eta) \big )
\eee
while for  $0<\gamma\le 1/2$, we add the extra correction term setting 
\bee\tw^{\pm}(\eta,\tau):=\hat w_0(\eta) + e^{-2\gamma\tau}\, \big ( \hat w_1(\eta)+\theta^{\pm}\tw_2(\eta) \big) +\Sigma_{k=2}^{N}\Sigma_{l=0}^k c_{k,l}  e^{-2k\gamma \tau} v_{k,l} (\eta).
\eee
Propositions \ref{prop-outer} and \ref{prop-outer2}, show that there exist $\tau_0$ and $\xi_0>0$,  such that $\tw^+$ and $\tw^-$ are super and sub solutions, respectively  on the region $(\eta,\tau)\in (A+\xi_0e^{-\gamma\tau},\infty) \times [\tau_0,\infty)$. 
Also,  following the results in the previous section \ref{sec-inner}, we define   the prospective super and sub-solutions $\bw^+_\e$ and $\bw^-_\e$ in the {\em inner region} by setting \begin{equation}\begin{aligned}\label{eq-w+w-}
&\bw^+_\e (\xi,\tau):= \fr{1}{1+\e} \, \bw_0 (\xi+C_1(\tau))\\
&\bw^-_\e(\xi,\tau):=\fr{1}{1-\e} \, \bw_0(\xi+ C_2(\tau)). 
 \end{aligned}\end{equation}
The   small  constant $\e\in[0,1)$ will be chosen later.  Also,   for some fixed $\xi_1$ to be determined later,  let  $C_1(\tau), \, C_2(\tau)$ be smooth functions  defined on $\tau \geq\tau_0$ such that 
 \be\label{eq-gluerelation}
e^{\gamma\tau}\, \tw^{\pm} (A+\xi_1\, e^{-\gamma\tau},\tau) = \bw^{\pm}_\e (\xi_1,\tau). 
\ee
Note that the  functions  $C_i(\tau)$ uniquely  exist  and are smooth because $\bar w_0(\cdot)$ is  strictly increasing smooth function onto $(0,\infty)$ 
and $\tw^{\pm}(A+\xi_1e^{-\gamma\tau},\tau)$  are positive smooth functions  on $\tau\geq \tau_0$. 
Moreover, since   $\tw_2 >0$  and  $\theta^+ > \theta^-$, we have  
$ e^{\gamma \tau}  \hat w^+(A+\xi_1 e^{-\gamma \tau},\tau) > e^{\gamma \tau} \hat 
w^-(A+\xi_1 e^{-\gamma \tau},\tau)$. Therefore    \eqref{eq-gluerelation} and the definition of $w^{\pm}_\e$ imply   that  $\bar w_0(\xi_1 + C_1(\tau)) > \bar w_0(\xi_1 + C_2(\tau))$.  Using again that $\bar w_0(\cdot)$ is a strictly increasing we conclude that
\be\label{eqn-CCC} C_1(\tau) > C_2(\tau), \qquad \tau \geq \tau_0
\ee
which  will be used later. 

It follows from the above discussion that for  $\tau \geq \tau_0$, we can glue  the functions $ e^{\gamma \tau}\, \tw^{\pm}(A+\xi e^{-\gamma \tau},\tau)$ and $\bw^{\pm}_\e(\xi,\tau) $   at 
$\xi=\xi_1$ to form a continuous and piecewise smooth function, namely we define 
\begin{equation}\begin{aligned}\label{eq-tildew1+}
&w^+_\e(\xi,\tau):=\begin{cases}\begin{aligned}
&\bar w^+_\e(\xi,\tau)&&\text{ if } \xi\le \xi_1\\
&e^{\gamma\tau}\tw^+(A+\xi e^{-\gamma\tau},\tau)&&\text{ if }\xi>\xi_1\end{aligned}\end{cases}
\\
& w^-_\e(\xi,\tau):=\begin{cases}\begin{aligned}
&\bar w^-_\e(\xi,\tau)&&\text{ if } \xi\le \xi_1\\
&e^{\gamma\tau}\tw^-(A+\xi e^{-\gamma\tau},\tau)&&\text{ if }\xi>\xi_1\end{aligned}\end{cases}
\end{aligned}\end{equation}  
(see Figure \ref{fig-2} below). 

We will show next that the functions $w^+_\e(\xi,\tau)$ and $w^-_\e(\xi,\tau)$ have the following properties:

\begin{prop}\label{prop-main1}
 There exist $\xi_1>0$ and $\e_1>0$ such that for any $0 < \e < \e_1$ there is a $\tau_1=\tau_1(\e)$ for which 
the functions 
 $w^+_\e$ and $w^-_\e$ given by   \eqref{eq-tildew1+} with  $0<\e< \e_1$,  have following properties:
\begin{enumerate}[{\em(i)}]
\item $w^+_\e(\xi,\tau)>w^-_\e(\xi,\tau)>0$ on  $(-\infty,\infty)\times [\tau_1,\infty)$;
\item $w^+_\e(\xi,\tau)$ and $w^-_\e(\xi,\tau)$ are continuous on $(-\infty,\infty)\times [\tau_1,\infty)$ and smooth  for $\xi \neq \xi_1$;
\item for all $(\xi,\tau)$ with $\xi \neq \xi_1$ and $\tau \geq \tau_1$, they satisfy  $I[w^+_\e]>0$ and $I[w^-_\e]<0$,  
i.e. they are super and sub-solutions, respectively; 
\item  at the non-smooth points  $(\xi_1,\tau)$, $\tau \geq \tau_1$, they satisfy  
$$\lim_{\xi\to\xi_1-}\fr{\partial}{\partial \xi}w^+_\e(\xi,\tau)>\lim_{\xi\to\xi_1+}\fr{\partial}{\partial \xi}w^+_\e(\xi,\tau)$$
$$\lim_{\xi\to\xi_1-}\fr{\partial}{\partial \xi}w^-_\e(\xi,\tau)<\lim_{\xi\to\xi_1+}\fr{\partial}{\partial \xi}w^-_\e(\xi,\tau).$$
\end{enumerate}
\end{prop}
\begin{figure}
\def\svgwidth{\linewidth}{\tiny
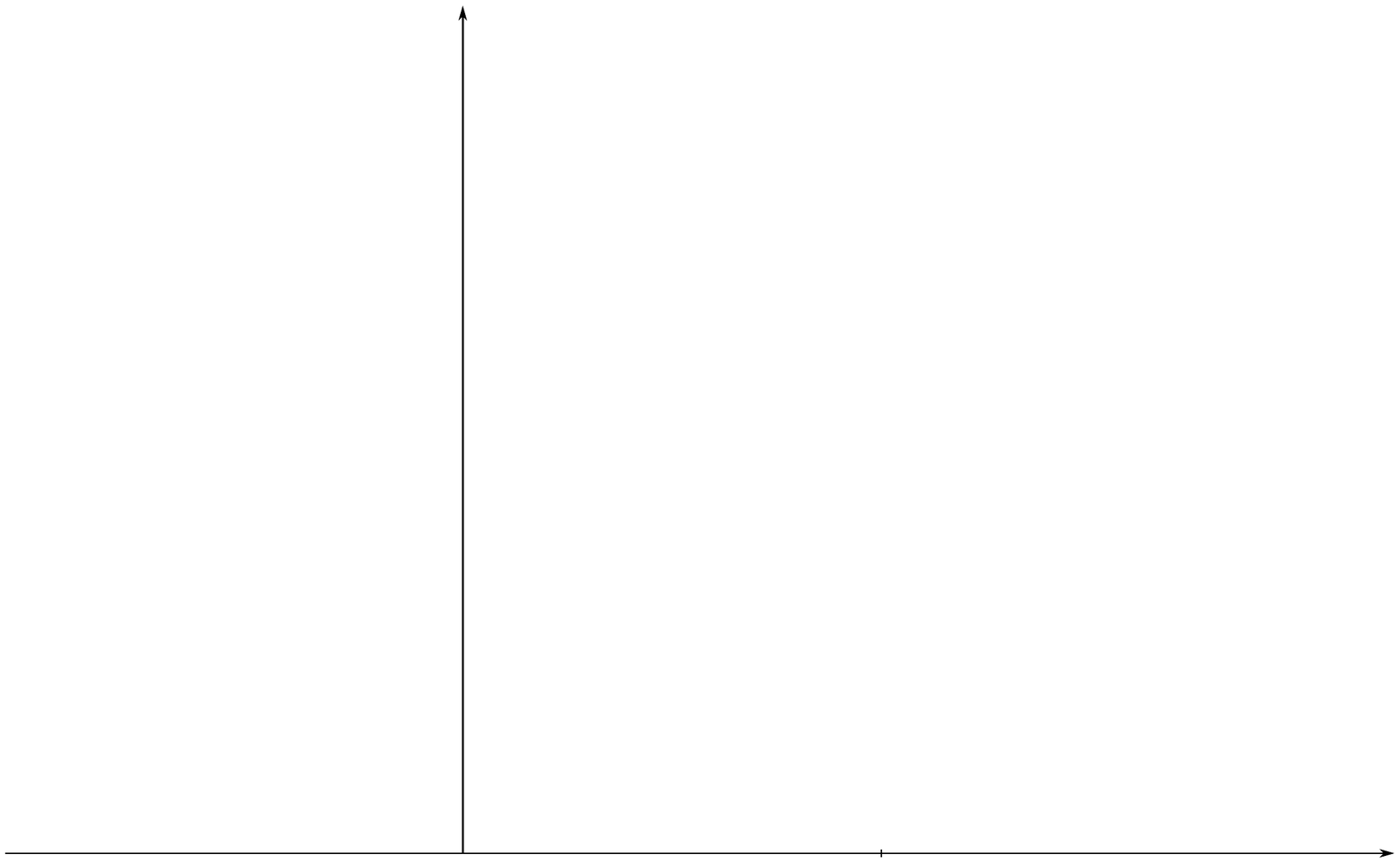
\caption{}\label{fig-2}}
\end{figure}

For the  proof of the proposition we will need the next two  lemmas. 

\begin{lemma}\label{lem-4}
For any fixed $\xi_1>\xi_0$, we have \begin{align}e^{\gamma\tau}\, \tw^\pm(A+\xi_1e^{-\gamma\tau},\tau)\to\fr{(n-1)(n-2)}{A\gamma} \xi_1 + \fr{(n-1)\theta^\pm}{A\gamma} \fr{1}{\xi_1}\end{align}and
\begin{align}\label{eqn-good1}\partial_\xi\Bigl[ e^{\gamma\tau}\tw^\pm(A+\xi e^{-\gamma\tau},\tau)\Bigr]_{\xi=\xi_1}\to\fr{(n-1)(n-2)}{A\gamma} - \fr{(n-1)\theta^\pm}{A\gamma} \fr{1}{\xi_1^2}\end{align}
as $\tau\to\infty$.
\begin{proof}
We have that $  { \lim_{\tau \to +\infty} e^{\gamma\tau} \tw_0(A+\xi_1e^{-\gamma\tau})= \fr{(n-1)(n-2)}{A\gamma} \xi_1}$  from Taylor's theorem on $\tw_0$ at $\xi=A$ and, for $\gamma >\fr{1}{2}$, $ {\lim_{\tau \to +\infty}e^{-\gamma\tau}\hat h(A+\xi_1e^{-\gamma\tau})}=\fr{(n-1)\theta}{A\gamma}\fr{1}{\xi_1}$ from Lemma  \ref{prop-tildew1at1}. This proves the first statement of the lemma for $\gamma >\fr{1}{2}$. Similarly, Taylor's Theorem on $\tw_0' $ and Lemma  \ref{prop-tildew1at1} imply the second statement for $\gamma>\fr{1}{2}$. 
In the case where $0 < \gamma \le1/2$ a statement similar to Lemma \ref{prop-tildew1at1} clearly holds for the modified $\hat h(\eta,\tau)$ and the result follows in the same way.

\end{proof} \end{lemma}

Although we haven't chosen  $\xi_1$ yet, we will next check that $w^+$ stays above $w^-$, for all small $\e>0$.

\begin{lemma}\label{lem-5}For any fixed $\xi_1$ and $\tau_0\in \R$, there exists $\e_0=\e_0(\xi_1,\tau_0)>0$ such that for all $0<\e<\e_0$, 
\begin{equation*}w^+_\e(\xi,\tau)>w^-_\e(\xi,\tau), \qquad \mbox{{\em for} }\,  (\xi,\tau) \in(-\infty,\infty) \times [\tau_0,+\infty). \end{equation*}

\begin{proof} Since  $\tw_2$ is a positive function, we have  $\tw^+>\tw^-$ and hence by the definition \eqref{eq-w+w-}
we have $$w^+_\e>w^-_\e\quad\text{ on }\ \xi\ge\xi_1\text{ and }\tau\ge\tau_0. $$ 
On the other hand, it  is obvious from the definition of $\bw^{\pm}_\e$ that \begin{equation}\label{eqn-we}
\bw^+_\e>\fr{1}{1+\e}\, \bw^+_0 \quad \text{ and } \quad \fr{1}{1-\e}\, \bw^-_0 > \bw^-_\e
\end{equation}
in the remaining region $\xi < \xi_1$ and $\tau \geq \tau_0$.  Here, $\bw^+_0$, $\bw^-_0$ are $\bw^+_\e$, $\bw^-_\e$ with $\e=0$.
Thus, it suffices to find small $\e_0>0$, depending on $\xi_1$,  such that \begin{equation}\label{eq-lem3-1}\fr{\bw^+_0(\xi,\tau)}{1+\e_0}>
\fr{\bw^-_0(\xi,\tau)}{1-\e_0}\qquad\text{ on }\ \xi\le\xi_1\ \, \text{ and } \ \, \tau\ge\tau_0.\end{equation}
To this end,  we will first show that if $C_{1,0}, C_{2,0}$ are defined by \eqref{eq-gluerelation} when $\e=0$, then $C_{1,0}>C_{2,0}$ and, as $\tau\to +\infty$,
\be\label{eqn-infty5}
C_{0,1}(\tau)\to C_{0,1,\infty}, \quad C_{0,2}(\tau)\to C_{0,2,\infty}  \quad \mbox{with} \quad  C_{0,1,\infty}>C_{0,2,\infty}.
\ee
Indeed,  this readily follows from the definition \eqref{eq-gluerelation}, $\tw^+ > \tw^-$ and the fact that  as $\tau \to +\infty$, 
$$e^{\gamma\tau}\tw^\pm(A+\xi_1e^{-\gamma\tau},\tau)\to\fr{(n-1)(n-2)}{A\gamma} \xi_1 + \fr{(n-1)\theta^\pm}{A\gamma} \fr{1}{\xi_1},$$ with  $\theta^+>\theta^-$. This in particular implies $\bar w_0^+ > \bar w_0^-$.

To conclude \eqref{eqn-we}, we will now use  \eqref{eqn-infty5}  and the fact that under the coordinate change \eqref{eq-uw} 
where $\xi=\ln r$ the 
functions $\bar w^{\pm}_0(\xi,\tau)$  are mapped into the functions $\bar U^{+}(r,\tau)>\bar U^{-}(r,\tau)$  given by 
$$\begin{aligned}&(\bar U^+)^{1-m}(r,\tau):= r^{-2} \bw^+_0(\ln r, \tau) = e^{2C_{0,1}(\tau)}\bar U^{1-m}(re^{C_{0,1}(\tau)})\\&(\bar U^-)^{1-m}(r,\tau) := r^{-2} \bw^-_0(\ln r, \tau)=e^{2C_{0,2}(\tau)}\bar U^{1-m}(re^{C_{0,2}(\tau)})\end{aligned}$$
where  under this  transformation   the region $\xi \leq \xi_1$ corresponds to the  compact region $\{x \in \R^n \, | \, r=|x| \le e^{\xi_1}\}$ 
of $\R^n$.  Here, recall that $\bar U^{1-m}(r) = r^{-2}\,  \bw_0(\ln r)$.
We then conclude, using \eqref{eqn-infty5} that there exists small $\e_0(\xi_1,\tau_0)$ such that $$\fr{1}{1+\e_0}(\bar U^+)^{1-m}>\fr{1}{1-\e_0}(\bar U^-)^{1-m} \quad\text{ on }r\le e^{\xi_1}\text{ and }\tau\ge\tau_0$$ showing   \eqref{eq-lem3-1}.
\end{proof}
\end{lemma}

We are now ready to proceed to the proof of Proposition \ref{prop-main1}.

\begin{proof}[Proof of Proposition \ref{prop-main1}] 
We have to find $\xi_1$, $\e_1$ and $\tau_1(\e)$ for each $0<\e<\e_1$. Notice that $\xi_0$ and $\tau_0$ come from Proposition \ref{prop-outer} and \ref{prop-outer2} and they are fixed throughout the proof. As long as $\tau>\tau_0$ and $0<\e<1$, $w^+_\e$ and $w^-_\e$ are well defined and part (ii) follows from their construction.  
For $\xi_1>\xi_0>0$ to be determined later, we have $\e_0(\xi_1)>0$ from Lemma \ref{lem-5} so that part (i) is true for $0<\e<\e_0$ and $\tau>\tau_0$. In summary, we may choose $\e_1\le \e_0(\xi_1)$ and any $\tau(\e)>\tau_0$ for undetermined $\xi_1>\xi_0>0$ so that part (i) and (ii) are always true. 
Before going to show (iv), let us recall asymptotic properties of $\bar w_0$ shown in \cite{CD}. As $\xi \to \infty $, we have \be\label{eq-barw_0}\begin{aligned} 
\bw_0(\xi)&= \fr{(n-1)(n-2)}{A\gamma}\xi+ \fr{(n-1)(n-6)}{4A\gamma}\fr{1}{\xi}+o(\xi^{-1})\\
\bw_0'(\xi)&= \fr{(n-1)(n-2)}{A\gamma} - \fr{(n-1)(n-6)}{4A\gamma}\fr{1}{\xi^2}+o(\xi^{-2}).
 \end{aligned}\ee

 Let us just check that $\lim_{\xi\to\xi_1-} ( w^+_\e)_{\xi} (\xi,\tau) 
>  \lim_{\xi\to\xi_1+} (w^+_\e)_{\xi} (\xi,\tau)$, as a similar argument holds for the other inequality.
 By the gluing condition and Lemma \ref{lem-4}, we have that for $\xi_1>\xi_0$, 
 $$w^+_\e (\xi_1, \tau) = (1+\e)^{-1} \bar w_0 (\xi_1+C_1(\tau))  \to\fr{(n-1)(n-2)}{A\gamma} \xi_1 + \fr{(n-1)\theta^\pm}{A\gamma} \fr{1}{\xi_1} $$
as $\tau \to \infty$. 
Let's assume $\e<1$. Invoking \eqref{eq-barw_0}, we may choose a large $\xi_1>\xi_0$  so that the following holds  independently from $\e$

\begin{itemize}
\item $\limsup_{\tau\to\infty}|C_1(\tau)- \e \xi_1| \le 1 $
\item $\liminf_{\tau\to\infty}(1+\e)^{-1}\bar w_0'(\xi_1+ C_1(\tau)) > \fr{(n-1)(n-2)}{(1+\e)A\gamma} - \fr{(n-1)}{(1+\e)A\gamma}\fr{2^{-1}(\theta^+ + \fr{n-6}{4})}{((1+\e)\xi_1)^2}.$ 
\end{itemize}
Continuing with this  choice of $\xi_1$, we may find small $\e_1>0$ so that $\e_1< \min(\e_0(\xi_1),1)$ and, for all $\e<\e_1$, $$ \fr{(n-1)(n-2)}{(1+\e)A\gamma} - \fr{(n-1)}{(1+\e)A\gamma}\fr{2^{-1}(\theta^+ + \fr{n-6}{4})}{((1+\e)\xi_1)^2}> \fr{(n-1)(n-2)}{A\gamma} - \fr{(n-1)}{A\gamma}\fr{\theta^+}{\xi_1^2}.$$
Since $$\lim_{\xi\to\xi_1-}\fr{\partial}{\partial \xi}w^+_\e(\xi,\tau) = (1+\e)^{-1}\bar w_0'(\xi_1+ C_1(\tau))$$ and $$\lim_{\xi\to\xi_1+}\fr{\partial}{\partial \xi}w^+_\e(\xi,\tau) = \partial_\xi\Bigl[ e^{\gamma\tau}\tw^+(A+\xi e^{-\gamma\tau},\tau)\Bigr]_{\xi=\xi_1},$$
the second part of Lemma \ref{lem-4} and above observation proves (iv) for a large $\tau_1$.

In showing (iii), we only need to check this in the inner region as we assume $\xi_1>\xi_0$ and $\tau_1>\tau_0$. By Proposition \ref{prop-inner}, 
it suffices to show for each fixed $\xi_1$ and $0<\e<\e_1$ there exists $\tau_1 \gg 1$ such that   
\be\label{eqn-C12-2}
|C'_1(\tau)| \leq M  \quad \mbox{and} \quad  |C'_2(\tau)| \leq M, \quad \mbox{for} \,\, \tau \geq \tau_0 
\ee
for some constant $M$. 
 We will actually show that $\lim_{\tau \to +\infty} C'_i(\tau) =0$, $i=1,2$ which 
yields \eqref{eqn-C12-2}. Lets prove  this for $C_1$, as the proof for $C_2$ is identical.
Recall   that 
\bee 0<e^{\gamma\tau}\, \tw^+(A+\xi_1e^{-\gamma\tau},\tau) = \bw^+_\e (\xi_1,\tau)=\fr{1}{1+\e}\bw_0(\xi_1+C_1(\tau))
\eee 
For  $\gamma>1/2$, 
differentiating  in  $\tau$  the LHS using that 
$e^{\gamma\tau}\, \tw^+=  e^{\gamma\tau} \, \tw_0 
+ e^{-\gamma\tau}\, \hat h$, we obtain
$$\mbox {LHS}  = \gamma \big ( e^{\gamma\tau}\, \tw_0-\xi_1\, \tw_0' \big )(A+\xi_1e^{-\gamma\tau})  -\gamma\, 
\big ( e^{-\gamma\tau}\hat h +\xi_1 e^{-2\gamma\tau}\hat h' \big ) (A+\xi_1e^{-\gamma\tau}).$$ 
Both  terms converge  to zero,  as $\tau\to\infty$,  by Taylor's theorem for $\tw_0$, $\tw_0'$  and the asymptotics in Lemma  \ref{prop-tildew1at1}.
The same convergence could be proven similarly for $0<\gamma \le 1/2$ as additional terms multiplied by $e^{\gamma\tau}$ are very small and their $\tau$-derivatives converges to zero at the point $(A+\xi_1e^{-\gamma\tau},\tau)$.   At the same time, for any $\gamma>0$, if we take derivative of RHS we obtain  
$$\mbox {RHS} = \fr{1}{1+\e}\bw_0'(\xi_1+C_1(\tau))\, C_1'(\tau).$$
Since the smooth function $C_1(\tau)$ converges as $\tau\to\infty$ and hence 
$$\bw_0'(\xi_1+C_1(\tau))\to \bw_0'(\xi_1+\lim_{\tau\to\infty}C_1(\tau))>0$$ 
this concludes that $C_1'(\tau)\to0$ as $\tau\to\infty$ and hence bounded for $\tau \gg 1$.  The same argument also applies to  $C_2(\tau)$.
Thus  \eqref{eqn-C12-2} holds.

Finally,  by the arguments above and  Propositions \ref{prop-outer}, \ref{prop-outer2} and  \ref{prop-inner}, we can find $\tau_1 \geq \tau_0$ which makes all the statements in our  proposition true.

\end{proof}

We will finish this section with the following result which is an immediate consequence of the comparison principle  and Proposition \ref{prop-main1}. 

\begin{thm}\label{thm-main1}
Let $\xi_1$, $\e_1$ and $\tau_1=\tau_1(\e)$ are such  Proposition \ref{prop-main1} holds. Assume that a  given conformally flat initial metric $g_0 = u_0^{1-m}(x) \, \delta_{ij} $ is bounded above and below by $w^+_\e(\xi,-\ln T)$ and $w^-_\e(\xi,-\ln T)$,  for some $0<\e<\e_1$ and $0<T<e^{-\tau_1}$,  via the coordinate change 
\be\label{eqn-changev}
w(\xi,\tau) = |x|^2 \, u(x,t), \quad \xi = \ln |x| - A\, e^{\gamma \tau}, \,\,\, \tau = - \ln (T-t)
\ee
at $t=0$. That is  
\be\label{eqn-ab0}
w^-_\e(\xi,-\ln T)\le \frac{|x|^2 u_0^{1-m}(x) }{T^{1+\gamma}} \le w^+_\e(\xi,-\ln T)
\ee
holds, with $\xi=\ln |x|-AT^{-\gamma}.$
Then, the solution of the Yamabe flow \eqref{eq-flatyamabe} exists on the time interval $(0,T)$ and it  is bounded between $w^+_\e(\xi,\tau)$ and $w^-_\e(\xi,\tau)$, that is 
\be\label{eqn-abt} w^-_\e(\xi,-\ln (T-t))\le \fr{ |x|^2\, u^{1-m}(x,t)}{(T-t)^{1+\gamma}} \le {w^+_\e(\xi,-\ln (T-t))}
\ee
with $\xi=\ln |x|-A(T-t)^{-\gamma}$. 

\begin{proof} Immediate by Proposition \ref{prop-main1} and  the comparison principle.\end{proof}\end{thm}
\section{Asymptotic shape of the singularity  in the inner region and  geometric properties}\label{sec-asymptoticshape}

Throughout this  section we will fix $\xi_1>0$ and $\e>0$  so that $w_\e^+(\xi,\tau)$   
and $w^-_\e(\xi,\tau)$ given by \eqref{eq-tildew1+} are barriers in view of Proposition \ref{prop-main1} and Theorem \ref{thm-main1}.  To simplify the notation we will  denote  them by $w^+(\xi,\tau)$ and $w^-(\xi,\tau)$ respectively. 
They are super and sub-solutions  of equation \eqref{eq-tildew} on $\R \times [\tau_1,\infty)$, respectively.

\smallskip 
We will first prove that  if  our initial conformally flat metric of the Yamabe flow \eqref{eq-flatyamabe} $u_0(\cdot)$ 
 is bounded from above and below by $w^+(\cdot,-\ln T)$ and $w^-(\cdot,-\ln T)$,  for some $-\ln T\ge \tau_1$ (c.f.
 \eqref{eqn-ab0}),   then the rescaled solution converges to a steady gradient soliton $\bar w_0(\xi)$, which is 
the unique entire solution of the equation \eqref{eqn-soliton1}
with asymptotic behavior \eqref{eqn-soliton2} as $\xi \to \infty$.

Since we are not  assuming  that our solution $u(x,t)$ of  \eqref{eq-flatyamabe}  is radially  symmetric,  it is more convenient to work in 
euclidean coordinates on $\R^n$, rather than cylindrical coordinates. 
We have seen that in order to see the steady state $\bar w_0$ in the {\em inner region} one needs to perform the coordinate change \eqref{eqn-cv2} on radially symmetric solutions in cylindrical coordinates. Under the transformation \eqref{eq-uw} which brings us back to the plane, this change of
variables corresponds the  coordinate change \eqref{eqn-u10}-\eqref{eqn-u11}  
which transforms a solution $u(x,t)$ of \eqref{eq-flatyamabe} to a solution $\bar u(y,l)$ of equation 
\be\label{eqn-ry1}
\partial_l \bar u -  \fr{1+\gamma}{(1-m)\gamma} \fr{\bar u}{l} =   \fr{n-1}{m}\Delta_y \bar u ^m + \gamma A (y \cdot \nabla_y \bar u) + \fr{2\gamma A}{1-m} \bar u.
\ee
We denote $\bar U(y)$ the steady soliton $\bar w_0$ in euclidean coordinates, namely 
$$\bar U(y)^{1-m}= |y|^{-2} \bw_0(\ln |y|).$$  This is the unique radial solution of 
\begin{equation}\label{eq-bu0} \fr{n-1}{m} \, \Delta u^m + \gamma A (y\cdot \nabla u )+ \fr{2\gamma A}{1-m}u =0
\end{equation} with asymptotic behavior 
\be\label{eqn-asy5}
u^{1-m}(y)=\fr{1}{|y|^2} \Bigl(\fr{(n-1)(n-2)}{\gamma A} \ln |y| + o(1)\Bigr).
\ee

We will next  prove the following result.

\begin{thm}\label{thm-main2} 
Under the assumptions  of Theorem \ref{thm-main1}, the rescaled solution $\bar u(y,l)$ converges, as $l \to +\infty$,  smoothly on compact sets of  $\R^n$   to the radial steady soliton $\bar U^{1-m}(y)$.

\begin{proof}
Let $l_0:= \gamma^{-1} T^{-\gamma}$ be the initial rescaled time, corresponding to $t=0$.   By Theorem \ref{thm-main1}, for  $l>l_0>0$ we have 
\begin{equation}
|y|^{-2}w^-(\ln |y|, \tau) \le \bar u^{1-m}(y,l) \le |y|^{-2} w^+(\ln |y|, \tau),\ l=\fr{e^{\gamma\tau}}{\gamma}.
\end{equation}
These two bounds give  upper and lower bounds away from zero for $\bar u (\cdot,l) $ on every compact set  in $\R^n$
which are uniform in time $l \geq  l_0 \gg 1$.  
Hence, by standard higher order regularity estimates for uniformly parabolic equations and a compactness argument, 
we conclude that for any sequence $l_i\to \infty$, the  solutions $\bar u_i(y,l):= \bar u (y, l_i+l)$ converge, passing to a subsequence,
 to  a limit $\bar u_\infty(y,l)$. The convergence is smooth on compact subsets of $\R^n\times \R$. Therefore, in view of \eqref{eqn-ry1} and the  uniform local upper bound of our sequence, the limit $\bar u_\infty$ is a smooth eternal solution of \begin{equation}\label{eq-Einfty}
\partial_l \bar u = \fr{n-1}{m}\, \Delta_y \bar u ^m + \gamma A (y \cdot \nabla_y \bar u) + \fr{2\gamma A}{1-m} \bar u.\end{equation}

To finish the proof we need to  show that  $$\bar u_\infty (y,l)=\bar U(y)$$which would also imply 
that our  limit is unique, thus  concluding that  $\bar u(\cdot,l)\to \bar U$,  as $l\to\infty$. 
To this  end,  we first observe that by  our barrier construction  \eqref{eq-tildew1+}, we have
\bee w^\pm(\ln |y|, \tau) \to \fr{(n-1)(n-2)}{\gamma A} \ln |y| + \fr{(n-1)}{A\gamma}\theta^\pm \fr{1}{\ln |y|}
\eee as $\tau \to +\infty$, uniformly on $e^{\xi_1}\le |y|\le K$,  for any fixed $K>e^ {\xi_1}$.
In particular, this implies that our limit $\bar u^{1-m}_\infty $ has these bounds and thus
\be\label{eq-uinftyasymp}
\bar u^{1-m}_\infty(\cdot, l)=\fr{1}{|y|^2} \Bigl(\fr{(n-1)(n-2)}{\gamma A} \ln |y| + o(1) \Bigr)\ee
as $|y|\to\infty$   uniformly in   $l \in \R$. 
For $\lambda>0$, if we denote ${\ds \bar U_\lambda(y) := \lambda^{\fr{2}{1-m}}\bar U(\lambda y)}$, this is again a radial solution of \eqref{eq-bu0} with \be\label{eq-bulambda}\bar U_\lambda^{1-m}(y)=\fr{1}{|y|^2} \Bigl(\fr{(n-1)(n-2)}{\gamma A} \ln |y| +\ln \lambda+ o(1)\Bigr).\ee
This is just a time translation of the radial steady soliton $\bar U$ and they are isometric. Since on the soliton the scalar curvature  $R>0$ everywhere, the solution pointwise decreases  as time increases  and hence $\bar U_{\lambda_1}>\bar U_{\lambda_2}$ for $\lambda_1>\lambda_2$. Thus we may define
\begin{align*}&\lambda_+:=\inf \, 
\{\lambda>0 \, \, | \, \bar U_\lambda(\cdot) \ge \bar u_\infty(\cdot, l)\text{ for all }l\}\\&\lambda_-:=\sup \{\lambda>0 \, \, | \, \bar U_\lambda(\cdot) \le \bar u_\infty(\cdot, l)\text{ for all }l\}.\end{align*} 
Our proof will finish   if we show that $\lambda_+=\lambda_-=1$. Let us prove that  $\lambda_+=1$. Since $\inf_{B(e^{\xi_1},0)}\bar U_\lambda\to \infty $ as $\lambda \to \infty$ (see  the observation in Corollary 3.3 in \cite{CD}), the   construction of $w^+$ in the inner region and \eqref{eq-bulambda} imply that we can find large $\lambda>1$ such that $\bar U_\lambda(\cdot) > \bar u_\infty(\cdot,l)$ for all $l$. By \eqref{eq-uinftyasymp} and \eqref{eq-bulambda}, $\bar U_\lambda(\cdot) \ngeq \bar u_\infty(\cdot,l)$ for $\lambda<1$. Therefore, $\lambda_+$ is a well defined number with $\lambda_+\ge1$. Assume that $\lambda_+>1$. For each $\bar U_{\lambda_+-2^{-n}}$, there is a  point $(x_n,l_n)$ with $\bar  U_{\lambda_+-2^{-n}}(x_n) < \bar U_\infty(x_n,l_n)$. Moreover, the sequence  of points $\{x_n\}$ 
 such that $\lambda_+-2^{-n}>1$  is bounded due to \eqref{eq-uinftyasymp} and \eqref{eq-bulambda}. By standard regularity estimates on the equation  \eqref{eq-Einfty}, we can find a subsequence of $(x_n,l_n)$ such that $$\bar u^*_{n_j}(x,l) := \bar u_\infty(x,l_{n_j}+l )\to \bar u^*_\infty(x,l)$$ smoothly on compact sets and  $x_{n_j}\to x^*$.

Note that  $\bar U_{\lambda_+}(x)\ge \bar u^*_\infty(x,l)$ for all $l$. 
On the other hand we have $\bar U_{\lambda_+}(x^*)= \bar u_\infty(x^*,0)$. Hence, by the strong maximum principle, 
we must have  $\bar U_{\lambda_+}(\cdot)=\bar u^*_\infty(\cdot,l)$,   for all $l$. But this   can't happen since \eqref{eq-uinftyasymp}, \eqref{eq-bulambda} holds and we have assumed that $\lambda_+>1$. 
By contradiction, this proves that $\lambda_+=1$ and $\lambda_-=1$ can be shown similarly. This concludes the proof of our theorem. 
\end{proof}
\end{thm}Let us remark the following. 
\begin{remark}[Scalar curvature blow up rate near the origin]  $u(x,t)$ of \eqref{eq-flatyamabe} represents conformally flat solution of the Yamabe flow  $g(t) =u^{1-m}(x,t)\, \delta_{ij} $  on $\R^n$. The metric of rescaled solution $\bar u(y,l)$ can be
written as  $$\bar u^{1-m} \delta_{ij}  = (T-t)^{-(1+\gamma)}\, \phi_t^* g(t)$$ where $\phi_t$ is a one  parameter family of diffeomorphisms 
$\phi_t(x)= e^{A(T-t)^{-\gamma}}\, x $. Therefore, Theorem \ref{thm-main2} 
can be  rephrased as convergence of the pointed manifold 
$$(\R^n, \fr{g(t)}{(T-t)^{1+\gamma}}, 0) \to (\R^n, \bar U^{1-m}\delta_{ij} , 0)$$ in Cheeger-Gromov sense.
This, in particular, implies that 
$$ |{\operatorname{Rm}}_{g(t)/(T-t)^{(1+\gamma)} } |(e^{A(T-t)^{-\gamma}}y) \to |{\text{Rm}}_{\bar U^{1-m}\delta_{ij}  } | (y)$$
 concluding that 
$$(T-t)^{1+\gamma}\, |\text{Rm}_{g(t)}| (e^{A(T-t)^{-\gamma}}y) \to |\text{Rm}_{\bar U^{1-m}\delta_{ij}  }| (y)$$
and also 
$$(T-t)^{1+\gamma} \, R_{g(t)} (e^{A(T-t)^{-\gamma}}y)\to R_{\bar U^{1-m}\delta_{ij}  } (y).$$
In particular the above implies the following blow up rate of the Riemmannian   and Scalar curvature at the origin
 $$\lim_{t\to T-} (T-t)^{1+\gamma} |\text{Rm}_{g(t)}|(0)= |{\text Rm}_{\bar U^{1-m}\delta_{ij} }| (0)=\fr{2\gamma A}{\sqrt{n(n-1)}}$$ and $$\lim_{t\to T-} (T-t)^{1+\gamma} \, R_{g(t)}(0)= R_{\bar U^{1-m}\delta_{ij} } (0)=2\gamma A.$$ 
\end{remark}

%

We will next show that the global supremum of the curvature occurs asymptotically {\em  at the  origin } 
 as $t\to T-$. 

\begin{prop}\label{prop-curvature} Under the assumptions of Theorem  \ref{thm-main1} 
we have 
\begin{equation}
\lim_{t \to T-} \Bigl[(T-t)^{1+\gamma} \sup_{x\in \R^n} |\mbox{\em  Rm}_{g(t)}|\Bigr]= \fr{2\gamma A}{\sqrt{n(n-1)}}
\end{equation}
and if  $\{(x_i,t_i)\}$ are points such that $t_i \to T$ and \begin{equation}
\lim_{i \to \infty} \Bigl[(T-t_i)^{1+\gamma} |\mbox{\em  Rm}_{g(t_i)}|(x_i)\Bigr]= \fr{2\gamma A}{\sqrt{n(n-1)}}
\end{equation} then $$\text{$y_i:=e^{-A(T-t_i)^{-\gamma}}x_i\to 0, \quad \mbox{\em as} \,\, i \to \infty$.}$$ 
In other words, $(T-t)^{-\fr{1+\gamma}{2}}\operatorname{dist}_{g(t_i)}(x_i,0)\to 0$ due to the convergence of the metric.
\end{prop} 

The result of this proposition follows from  the following curvature estimate lemma in the outer region. This lemma is useful in the sense that it also gives the  curvature blow up rate in  other regions. For example, it shows  that  the curvature blows up in a {\em type I manner}   near the {\em infinite cylindrical region}.
\begin{lemma}\label{lem-curvature}
There exist  $C>0$,  $r_0 \gg 1$ and $0<t_0<T$ such that  the following holds 
\begin{equation}
(T-t) \, \hat w_0 \big (A+ (T-t)^{\gamma}\ln |y| \big )  \, |\mbox{\em  Rm}_{g(t)}(e^{A(T-t)^{-\gamma}} y)| \leq C\end{equation}
on   $|y| \geq r_0$ and $ t_0<t<T$.
\end{lemma}

Let us first show that Lemma \ref{lem-curvature} implies the proposition and then finish this section by proving  the lemma.
\begin{proof}[Proof of Proposition \ref{prop-curvature}]

Assuming that the Lemma \ref{lem-curvature}  holds,  then since $\hat w_0$ is increasing function, for each fixed $r_1 \geq r_0 $ we have 
 $$(T-t)^{1+\gamma} \sup_{|y|\ge r_1} \, |\text{Rm}_{g(t)}(e^{A(T-t)^{-\gamma}} y)| < \fr{C(T-t)^{\gamma}}{\hat w_0(A+ (T-t)^{\gamma}\ln |r_1|)}
 $$ and, by the  Taylor  expansion  of $\hat w_0(\cdot)$ at $A$,  taking  the limit $t\to T-$,  yields   
\be\label{eq-113}\limsup_{t\to T-}\Bigl[ (T-t)^{1+\gamma} \sup_{|y|\ge r_1} |\text{Rm}_{g(t)}(e^{A(T-t)^{-\gamma}} y)|\Bigr]\le \fr{C\gamma A}{(n-1)(n-2) \ln r_1}.\ee
Choose    $r_1  \geq  r_0$ sufficiently large  so that 
\[{\ds \fr{C\gamma A}{(n-1)(n-2) \ln r_1} <  \fr{2\gamma A}{\sqrt{n(n-1)}}}.\]  Now on the remaining region   $|x| \le r_1$, 
 due to the smooth convergence of $\bar u^{1-m}$ to $\bar U^{1-m}$ on compact sets and the fact that the steady soliton $\bar U ^{1-m} $ attains its maximum curvature at the origin we have  
 $$\lim_{t\to T-} \Bigl[ (T-t)^{1+\gamma} \sup_{|y|\le r_1} |Rm_{g(t)}(e^{A(T-t)^{-\gamma}} y)|\Bigr]= \fr{2\gamma A}{\sqrt {n(n-1)}}$$ 
 and also  the second statement of the proposition holds, concluding the proof.  
 
\end{proof} 

\smallskip 
\begin{proof}[Proof of Lemma \ref{lem-curvature}] For $r_0$ and $t_0$ to be chosen later, we pick and fix a point $(y_1, t_1)$ with $|y_1| \geq r_0$ and $t_0 < t_1 <T$ and consider the following scaling of the solution \bee
v^{1-m}(z,\sigma):= (e^{A(T-t_1)^{-\gamma}}|y_1|)^2 \fr{u^{1-m}(e^{A(T-t_1)^{-\gamma}} |y_1|\,  z,t)}{(T-t)\, \hat w_0 (A+(T-t_1)^\gamma \ln|y_1|)}
\eee
with \begin{equation}\label{eqn-sigma}  \sigma(t):= \ln \big ( \fr{T-t_1}{T-t}\big ) \,  \fr{1}{\hat w_0(A+ (T-t_1)^{\gamma} \ln |y_1|)}.
\end{equation}
Since  $u$ satisfies  \eqref{eq-flatyamabe}, it follows that $v(z,\sigma)$ evolves by 
\begin{equation}\label{eq-v}
\fr{\partial }{\partial \sigma} v = \fr{n-1}{m} \Delta_z v^m - \hat w_0(A+(T-t_1)^\gamma \ln |y_1|)\,  v.
\end{equation}

\smallskip 
\begin{claim}\label{claim-10}
There are  $r_0>1$ and $t_0>0$ so that if $|y_1|\ge r_0$ and $t_1\ge t_0$, then there exist positive constants $c$ and $C$ 
which are  independent of $(y_1,t_1)$ such that $$c  \le v^{1-m}(z,\sigma) \le C, 
 \qquad \mbox{for}\,\,\, (z,\sigma)\in \mathcal{A}\times[-1,0]$$ 
where  $\mathcal{A}$ is the annulus $\mathcal{A}=\{z\in\R^n \,\, | \,\, 1/2 \le|z|\le 3/2 \, \}$.
 \end{claim} 

Lets us assume that the claim holds and finish the proof of the lemma. Since $|\hat w_0 |\le(n-1)(n-2)$, the equation \eqref{eq-v} is uniformly parabolic
on $\mathcal{A}\times[-1,0]$,  independently from   choice  of 
the point $(y_1,t_1)$,  and therefore by standard parabolic regularity estimates, we have uniform bounds for
 $|\nabla v| $ and $|\nabla^2 v|$ on any strictly smaller parabolic cylinder. In particular, we have 
 $$|\nabla v (y_1/|y_1|, 0)|  \quad\mbox{and} \quad |\nabla^2  v (y_1/|y_1|, 0)| <C.$$
Since $v^{1-m}(\cdot,0)\delta_{ij} $ and $u^{1-m}(\cdot,t_1)\delta_{ij} $ are isometric, we conclude that 
\bee
\begin{split}
(T-t_1) \, \hat w_0(A+ (T-t_1)^{\gamma}\ln |y_1|) \,  \big  | {\text {Rm}}_{g(t_1)} &(e^{A(T-t_1)^{-\gamma}} y_1) \big  |\\
&=   |{\text {Rm}}_{v^{1-m}(0)\delta_{ij} } \big (\fr{y_1}{|y_1|} \big ) \big  | \leq C.
 \end{split}
 \eee
The constant $C$ is  independent from  $(y_1,t_1)$.  This finishes  the proof of the lemma.

\begin{proof}[Proof of the Claim \ref{claim-10}] Although the computations below might look intimidating, the idea is simple.  Since $u$ is trapped  between the two barriers
in the outer region  where  $ |x| \ge e^{\fr{A}{(T-t)^{\gamma}}}e^{\xi_1}$, namely 
\be\label{eq-111} \hat w^-((T-t)^{\gamma}\ln |x|, -\ln (T-t)) \le \fr{|x|^2u(x,t)}{T-t} \le \hat w^+( (T-t)^{\gamma}\ln |x|, -\ln (T- t)) \ee 
and  since  $\hat w^+$ and $\hat w^-$ are close to $\hat w_0$,  we  expect that different values of $v$ are similar  in the whole annulus.

Indeed, suppose $r_0 >  1$  and $t_0 \in (0,T)$ are first chosen to satisfy
\bee\label{eq-r0t0}\text{ $\ln \fr{r_0}{2} >  \xi_1 \quad \mbox{and} \quad - \ln (T- t_0) >-\ln T+(n-1)(n-2)$}.\eee
With this choice of $t_0$ and  $0\le \hat w _0\le (n-1)(n-2)$, we see for that  \[\sigma(0)=\fr{\ln (T-t_1)-\ln (T-0)}{w_0(A+ (T-t_1)^{\gamma} \ln |y_1|)}\le \fr{\ln (T-t_0)-\ln T}{(n-1)(n-2)} \le -\fr{(n-1)(n-2)}{(n-1)(n-2)}= -1\] 
where $\sigma(t)$ is defined by \eqref{eqn-sigma}. Since $u(x,t)$ is defined for $t\ge0$, this shows that the rescaled function $v(z,\sigma)$  is well defined on $\mathcal{A}\times [-1,0]$. By choosing $r_0>2e^{\xi_1}$ sufficiently large  and $t_0 \in(0,T)$ closer to $T$, we may  assume that 
\be\label{eq-sec7-1}
\fr{1}{2}  \tw_0( A+ \xi e^{-\gamma \tau}) \le \tw^-(A+ \xi e^{-\gamma \tau},\tau)
\ee
and \be\label{eq-sec7-2} \tw^+(A+ \xi e^{-\gamma \tau}, \tau) \le 2 \tw_0( A+ \xi e^{-\gamma \tau})\ee
on $\xi \ge \ln  (r_0/2) $ and $\sigma\ge-1$. This is possible because $\hat w^\pm = \hat w_0 + e^{-2\gamma t} \hat h$,  where $\hat h$ is bounded away from $A$ and satisfies \ref{prop-tildew1at1} near $A$,  when $\gamma >1/2$. The other  range $\gamma\in(0,1/2]$ is   similar (see Claim \ref{claim-outer4}).
Using  \eqref{eq-111} we can then estimate
\begin{align*}
v^{1-m}(z,\sigma) &= \fr{(e^{\fr{A}{(T-t_1)^{\gamma}}}|y_1|)^2}{T-t} \fr{u^{1-m}(e^{\fr{A}{(T-t)^{\gamma}}}|y_1| \, e^{\fr{A}{(T-t_1)^{\gamma}}-\fr{A}{(T-t)^{\gamma}} } z,t)}{\hat w_0 (A+ (T-t_1)^\gamma \, \ln |y_1|)}\\
&\le\fr{1}{|z|^2}\fr{ \tw^+ \Bigl (A+  (T-t)^\gamma \ln \big ( |y_1||z|e^{\fr{A}{(T-t_1)^{\gamma}}-\fr{A}{(T-t)^{\gamma}} } \big ),\tau
\Bigr )}{\hat w_0 (A+ (T-t_1)^\gamma \, \ln |y_1|)}\\
&\le \fr{2}{|z|^2}\, \fr{\tw_0 \Bigl (A+  (T-t)^{\gamma} \, \big (  \ln |y_1| + \ln |z|\big ) +A \, \big  ( \fr{(T-t)^\gamma}{(T-t_1)^{\gamma}} -1 \big  ) \Bigr)}{\tw_0 (A+ (T-t_1)^\gamma \ln |y_1|)}.
\end{align*}
Since $t \leq t_1$ and $|y_1| \geq  r_0$, we have ${\ds  |y_1||z|e^{\fr{A}{(T-t_1)^{\gamma}}-\fr{A}{(T-t)^{\gamma}} } 
\ge{ r_0}\cdot  (1/2)\cdot  1\ge e^{\xi_1}}$, thus  we could  bound  above $v^{1-m}(z,\sigma)$ using our  barrier $\hat w^+$ of the outer region in the second line and use \eqref{eq-sec7-2} in the last line.
Similarly we get a lower bound 
\bee
v^{1-m}(z,\sigma) \geq \fr{1}{2|z|^2}\fr{\tw_0 \Bigl (A+  (T-t)^{\gamma} \, \big (  \ln |y_1| + \ln |z|\big ) +A \, \big  ( \fr{(T-t)^\gamma}{(T-t_1)^{\gamma}} -1 \big  ) \Bigr)}{\tw_0 (A+ (T-t_1)^\gamma \ln |y_1|)}.
\eee
If we could find constants   $0<c<1<C$ such that 
\begin{equation}\label{eq-bound} c\le\fr{(T-t)^{\gamma} \, \big (  \ln |y_1| + \ln |z|\big ) +A \, \big  ( \fr{(T-t)^\gamma}{(T-t_1)^{\gamma}} -1 \big  )}{(T-t)^{\gamma}\, \ln |y_1|}\le C
\end{equation} 
holds on $(v,\sigma)\in A\times [-1,0]$, then since  $\tw_0(A+x)$ is an  increasing function on $x\ge0$, concave and $\tw_0(A+0)=0$, we would deduce
that $$ c \le \fr{\tw_0 \Bigl (A+  (T-t)^{\gamma} \, \big (  \ln |y_1| + \ln |z|\big ) +A \, \big  ( \fr{(T-t)^\gamma}{(T-t_1)^{\gamma}} -1 \big  ) \Bigr)}{\tw_0 (A+(T-t)^{\gamma} \, \ln |y_1|)} \le C$$
which combined with the above would finish  the proof of our claim. 

Now let us now find the bounds \eqref{eq-bound}.  First, it is clear that 
\be\label{eqn-ine1}
1+\fr{ \ln \fr{1}{2}}{\ln r_0} \le \fr{\ln |y_1|+\ln |z|}{\ln |y_1|} \le 1 + \fr{\ln \fr{3}{2}}{\ln r_0}
\ee and notice that  ${\ds 1+ \fr{ \ln \fr{1}{2}}{\ln r_0} = \fr{\ln \fr{r_0}{2}}{\ln r_0} }$ is positive since ${\ds \ln \fr{r_0}{2} > \xi_1 >0}$. Thus, it suffices to prove
\begin{equation}\label{eqn-ine2} 0 < \fr{A \, \big  ( \fr{(T-t)^\gamma}{(T-t_1)^{\gamma}} -1 \big  )}{(T-t)^{\gamma}\, \ln |y_1|}<C_1
\end{equation} 
by some constant $C_1 < \infty$. If we introduce ${\theta:= (T-t_1)^\gamma} \, \ln |y_1|>0$, then by the definition of $\sigma$ in \eqref{eqn-sigma}
$$ \fr{(T-t)^\gamma}{(T-t_1)^{\gamma}}  = e^{\ds -\sigma \gamma \, \tw_0 (A+ \ln |y_1|(T-t_1)^\gamma )}=e^{\ds -\sigma \gamma \, \tw_0 (A+ \theta)}
>0.$$ Now $0\le -\sigma\le 1$ and ${\ds  e^{-\sigma \gamma \tw_0(A+\theta)}\ge1}$ imply  $$ 0< \fr{A \Bigl ( \fr{(T-t)^\gamma}{(T-t_1)^{\gamma}} -1 \Bigr )}{(T-t)^{\gamma}\, \ln |y_1|  }= A\, \fr{e^{-\sigma\gamma \tw_0 \, (A+\theta)}-1}{\theta \, e^{-\sigma \gamma \tw_0\, (A+\theta)}}\le A\, \fr{e^{\gamma \tw_0\, (A+\theta)} - e^{\gamma \tw_0\, (A+0)}}{\theta-0}. $$
By the  mean value theorem, there exists $0 < \theta_0 < \theta $ such that  
\bee\begin{split}
\fr{e^{\gamma \, \tw_0(A+\theta)} - e^{\gamma \, \tw_0(A+0)}}{\theta-0} &= \gamma \, \tw_0'(A+\theta_0) \, e^{\gamma \tw_0(A+\theta_0)}\le \gamma \, \fr{(n-1)(n-2)}{\gamma A} \, e^{\gamma (n-1)(n-2)}.
\end{split}
\eee
Here, we used the facts  that $\tw_0(A+x)$ is a  nonnegative concave function on the set  $x\ge0$ with \[\text{${\ds \tw'_0(A)=\fr{(n-1)(n-2)}{\gamma A}} \quad $ and $\quad {\ds\tw_0(A+x)\le{(n-1)(n-2)}}$}.\] Combining the last two inequalities  implies that \eqref{eqn-ine2} holds with \[C_1:= {(n-1)(n-2)} \, e^{\gamma (n-1)(n-2)}.\] 
This finished the proof of the claim and also the proof of the lemma. 
\end{proof}

\end{proof}

\section{Proof of Theorem \ref{thm-main}}\label{sec-new}

In this final  section, we will give the proof of Theorem \ref{thm-main} as stated in the introduction.  We will need the following rigidity result for eternal solutions of conformally flat Yamabe flow.
\begin{prop}\label{prop-rigidity}Let $g(t)=u^{1-m}(x,t)\, \delta_{ij} $   be a smooth eternal solution to the conformally flat Yamabe flow \eqref{eq-flatyamabe} on $\R^n\times (-\infty,\infty)$, with positive Ricci and uniformly bounded sectional curvature. We further assume that $u$ is bounded from below by a radial steady gradient soliton centered at the origin with maximum scalar curvature $2\gamma A>0$, that is   \be  u(x,t)\ge \label{eq-eternalbounds}e^{-\fr{2\gamma A(t+\xi_1)}{1-m}}\bar U(|x|e^{-\gamma A(t+\xi_1)})\quad \text{ for some }\xi_1\in\R.\ee
Then, $u(x,t)$ must be a radial  gradient steady soliton, that is  \be u(x,t)\equiv e^{-\fr{2\gamma A(t+\xi_0)}{1-m}}\bar U(|x|e^{-\gamma A(t+\xi_0)})\quad\text{  for some }\xi_0 \le \xi_1.\ee

\begin{proof}
By the Harnack inequality for the Yamabe flow (c.f.  Theorem 3.7 in \cite{Ch}),   for any 1-form $X_i$ \be \label{eq-harnackexpression}(n-1)\Delta R +\langle  \nabla R,X\rangle + \fr{1}{n-1}R_{ij}X^iX^j +R^2\ge 0.\ee 
Note that since our solution exists from $t=-\infty$ we  could drop $R/t$ term from the original Harnack expression  in \cite{Ch}.
 This inequality, in particular implies that $$(n-1)\Delta R +R^2 =\partial_t R>0.$$
\begin{claim}\label{claim-cc}$R_g(0,0) = 2\gamma A$ and $2\gamma A = \sup R_g(x,t)$. 
\end{claim}
\begin{proof}[Proof of Claim \ref{claim-cc}]The proof is simple  and  uses that $\partial_t R>0$  and  \be \partial_t u^{1-m}(x,t)= -R_g(x,t) \, u^{1-m}(x,t).\ee  Suppose there is a point $(x_0,t_0)$ with $R_g(x_0,t_0)>2\gamma A$. Since $\partial_t R>0$, a ODE comparison implies$$u(x_0,t)= C(x_0) \,  e^{-\fr{1}{1-m}R_g(x_0,t_0)t}, \quad \text{ as }t\to \infty.$$ On the other hand, this contradicts to $$u(x_0,t)\ge \bar e^{-\fr{2\gamma A(t+\xi_1)}{1-m}}\bar u_0(|x_0|e^{-\gamma A(t+\xi_1)})\ge e^{-\fr{2\gamma A(t+\xi_1)}{1-m}}\inf_{{|y|\le|x_0|e^{-\gamma A(t_0+\xi_1)}}}\bar u_0(y)$$ which holds for $t\ge t_0$. Similarly as before, if  $R_g(0,0)<2\gamma A$, then $$u(0,t)= C_0 \, e^{-\fr{1}{1-m}R_g(0,0)t}, \quad \text{ as }t\to -\infty $$ which  again contradicts to \eqref{eq-eternalbounds}.
\end{proof}

According to the classification of conformally flat radial solitons (c.f.  Propositions 1.4 and 1.5 in  \cite{DS}) the  one parameter family of solutions 
\[\bar U_{\xi}(x,t):= e^{-\fr{2\gamma A(t+\xi)}{1-m}}\bar U(|x|e^{-\gamma A(t+\xi)}), \qquad \xi\in \R \]  are all possible conformally flat radial steady gradient solitons  whose maximum scalar curvature is $2\gamma A$ at the origin. It also is known that these solutions attain a strict  curvature maximum at the origin.  Meanwhile, due to the Claim \ref{claim-cc}, $g(x,t)$ attains its maximum scalar curvature  at an interior space-time point $(0,0)$. Furthermore, since $g(x,t)$  has positive  Ricci curvature and  bounded sectional curvature, Corollary 5.1 in \cite{DS} implies that $g(x,t)$ must be a steady gradient soliton. Also, under the nonnegative Ricci condition, such steady gradient solitons are (globally) conformally flat and radially symmetric (c.f.  Theorem 3.2 and Corollary 3.3 in \cite{CMM} or Corollary 1.6 and Remark 1.2 in \cite{CSZ}).  Since $g(x,t)$ has its maximum scalar curvature at the origin, it must be symmetric with respect to this point. In view of Liouville's rigidity theorem on conformal mappings on $\R^n$ with $n\ge3$, $u(x,t)$ must be a radially symmetric function which represents a steady gradient soliton. Hence $u(x,t)=\bar U_{\xi_0}(x,t)$ by the classification theorem in \cite{DS}. 
\end{proof}
\end{prop}

\medskip
We are now in position to finally give the proof of our main result, Theorem \ref{thm-main}. 

\begin{proof}[Proof of Theorem \ref{thm-main}]
We begin by fixing  $\xi_1$, $\e\le \e_1$ and $\tau_1$,  as they appear in  Proposition \ref{prop-main1} and Theorem \ref{thm-main1}. Set $T_1:=e^{-\tau_1}$
and fix $T$ with $0 < T < T_1$.   

\begin{claim}\label{claim-cc2}
There exist $\xi_a>\xi_b$ so that   
\be \label{eq-boundch8}w^+_\e(\xi-\gamma A \xi_a,-\ln T)\le \frac{|x|^2 u_0^{1-m}(x) }{T^{1+\gamma}} \le w^-_\e(\xi-\gamma A\xi_b,-\ln T)\ee
holds, under the  coordinate change $\xi=\ln |x|-AT^{-\gamma}.$
\end{claim}
\begin{proof}[Proof of Claim \ref{claim-cc2}] For our fixed $T$, it is not hard to check  that $$T^{1+\gamma}\, w^\pm_\e (\xi, -\ln T) = (n-1)(n-2)\left( T- (\xi/A)^{-\fr{1}{\gamma}}+O(\xi^{-\fr{1}{\gamma}-1})\right), \quad \mbox{as } \,\, \xi\to \infty$$  and $$T^{1+\gamma}w^\pm_\e (\xi, -\ln T) \sim e^{2\xi}, \quad \mbox{as}\,\, \xi\to -\infty.$$

We can also check that if $\ln |x| = \xi + AT^{-\gamma}$,  $|x|^2 u^{1-m}_0(x)$ also satisfies  these asymptotics. Indeed, Condition ii)  in Theorem 1.2 implies that
\[ \begin{aligned}|x|^2u_0^{1-m}(x) &= {(n-1)(n-2)}\,  \left ( T- \left(\fr{\xi +AT^{-\gamma}}{A}\right)^{-\fr{1}{\gamma}} + O\big ((\xi +AT^{-\gamma})^{-\fr{1}{\gamma}-1}\big)\right)\\
&= (n-1)(n-2)\left( T- (\xi/A)^{-\fr{1}{\gamma}}+O(\xi^{-\fr{1}{\gamma}-1})\right)\quad \mbox{as}\,\,\xi\to \infty\end{aligned} \]
and also 
 \[ \lim_{\xi\to -\infty} \fr{|x|^2u^{1-m}_0(x)}{e^{2\xi}} =\lim_{\xi\to-\infty} e ^{AT^{-\gamma}}u^{1-m}_0(x) =  e ^{AT^{-\gamma}}\, u^{1-m}_0(0)>0. \]
Using these  asymptotic behaviors, we can first find $\xi_{a,0}>\xi_{b,0}$ such that \eqref{eq-boundch8} holds asymptotically (outside of compact interval in $\xi$). Next, we may use condition i) of Theorem 1.2 to find possibly smaller $\xi_b\le\xi_{b,0}$ so that the second inequality of \eqref{eq-boundch8} holds everywhere. Finally, by a similar argument which uses the fact $|x|^2u_0(x)$ is uniformly bounded away from zero on $|x|\ge r_0$ for all $r_0>0$, we may find larger $\xi_a\ge \xi_{a,0}$ so that the first inequality of \eqref{eq-boundch8} holds everywhere. This argument is very similar to the proof of Claim 4.4 in \cite{CD}.
\end{proof}
Let $\bar u(y,l)$ be the  rescaled  solution obtained from $u(x,t)$ under  \eqref{eqn-u10}-\eqref{eqn-u11}.  By Theorem \ref{thm-main1},  \ref{thm-main2} and the claim, 
we have local uniform upper and lower bounds on  $\bar u$, namely $\bar u_a \le \bar u \le \bar u_b$ where  
 $$\bar u_a(y,l) \to e^{-\fr{2\gamma A \xi_a}{1-m}}\bar U (ye^{-\gamma A\xi_a})=\bar U_{\xi_a} (y,0)$$
 and
 $$\bar u_b(y,l) \to e^{-\fr{2\gamma A \xi_b}{1-m}}\bar U (ye^{-\gamma A\xi_b})=\bar U_{\xi_b} (y,0)$$ 
 locally uniformly in $y$ as $l\to \infty$.

In order to show the  blow up rate \eqref{eq-blowuprate}, we first need the following claim which asserts that Lemma \ref{lem-curvature} holds for our given  solution. 
%

\begin{claim}\label{claim-sec8-2} For our metric  $g(x,t)=u^{1-m}(x,t) \, \delta_{ij}  $, there exist  $C>0$,  $r_0 \gg 1$ and $0<t_0<T$ such that  the following holds 
\begin{equation}
(T-t) \, \hat w_0 \big (A+ (T-t)^{\gamma}\ln |y| \big )  \, |\mbox{\em  Rm}_{g(t)}(e^{A(T-t)^{-\gamma}} y)| \leq C\end{equation}
on   $|y| \geq r_0$ and $ t_0<t<T$.
\end{claim}\begin{proof} [Proof of Claim \ref{claim-sec8-2}] The proof is  the same as that of  Lemma \ref{lem-curvature} except a few modifications
which we point out next. Instead of  \eqref{eq-111}, now we have \bee 
\hat w^-(\Xi_a, -\ln (T-t)) \le \fr{|x|^2u(x,t)}{T-t} 
\le \hat w^+( \Xi_b, -\ln (T- t))
\eee
with $\Xi_a:= (T-t)^{\gamma}(\ln |x|-\gamma A \xi_a)$, $\Xi_b:=(T-t)^{\gamma}(\ln |x|-\gamma A\xi_b)$ and  for those points $(x,t)$ with $ |x|e^{-\gamma A \xi_a} \ge e^{\fr{A}{(T-t)^{\gamma}}}e^{\xi_1}$.

The proof of Lemma \ref{lem-curvature} now applies  after if  choose  possibly larger $r_0$ and $t_0$ now  depending on $\xi_a$ and $\xi_b$. To be specific, we can choose them $r_0$ and $t_0$ so that $\ln \fr{r_0}{2} \ge \xi_1 +\gamma A \xi_a$ and we have following inequalities instead of \eqref{eq-sec7-1} and \eqref{eq-sec7-2} 
\bee\label{eq-sec8-1}
\fr{1}{2}  \tw_0( A+ \xi e^{-\gamma \tau}) \le \tw^-(A+ (\xi-\gamma A \xi_a) e^{-\gamma \tau},\tau)
\eee
and \bee\label{eq-sec8-2} \tw^+(A+ (\xi-\gamma A \xi_b) e^{-\gamma \tau}, \tau) \le 2 \tw_0( A+ \xi e^{-\gamma \tau})\eee
on $\xi \ge \ln  (r_0/2) $ and $\sigma\ge-1$.
The rest of the proof follows as before.\end{proof}

\smallskip

We now continue with the proof of the theorem. Due to the claim, we have for $r_1\ge r_0$  (c.f. \eqref{eq-113}) \be\label{eq-114}\begin{aligned}\sup_{\R^n \setminus B_{r_1}(0)}  |\text{Rm}_{\bar g_\infty (l)}|&\le \limsup_{t\to T-}\Bigl[ (T-t)^{1+\gamma} \sup_{|y|\ge r_1} |\text{Rm}_{g(t)}(e^{A(T-t)^{-\gamma}} y)|\Bigr] \\ &\le \fr{C\gamma A}{(n-1)(n-2) \ln r_1}.\end{aligned}\ee

Let us consider a given sequence $l_i \to \infty$. Using  the  two bounds $\bar u_a $ and $\bar u_b$, we may pass to a subsequence $\bar u(y,l+l_i)$ and obtain a $C_{\operatorname{loc}}^\infty(\R^n\times\R)$ limit $\bar u_\infty$ which is an eternal solution  of the equation \eqref{eq-Einfty}. After taking  limit our two bounds imply  \be \label{eq-asympbounds}\bar U_{\xi_a} (y,0) \le \bar u_\infty(y,l) \le \bar U_{\xi_b} (y,0).\ee 
Now our limit $\bar g_\infty(y,l) = \bar u_\infty^{1-m}(y,l) \delta_{ij} $ has nonnegative Ricci since this  is preserved along the flow and the limit under the locally conformally flat condition (c.f.   \cite{Ch}).  

Our final step will be  to show that $\bar  u_\infty(y,l) $ must be one of the steady gradient solitons \[\bar U_{\xi_0}(y,0)= e^{-\fr{2\gamma A\xi_0}{1-m}}\bar U(|y|e^{-\gamma A\xi_0}).\] Note that the  time dilation parameter $\xi_0$ might be different for different limits along sequences  $l_i \to \infty$, but metrics with different  $\xi_0$ represent the same soliton and thus this proves  Cheeger-Gromov convergence of the metric $\bar u^{1-m}(y,l)\delta_{ij} $ to the same limit soliton as $l\to\infty$.  Also this convergence and  \eqref{eq-114} proves \eqref{eq-blowuprate}  (c.f.  Proposition \ref{prop-curvature}). 

Let us consider  $ u_\infty(y,l) :=  e^{-\fr{2\gamma A}{1-m} l} \bar u_\infty (ye^{-\gamma Al},l)$. Then \eqref{eq-asympbounds} turns into the inequality between eternal solutions of conformally flat Yamabe flow \eqref{eq-flatyamabe} \be\label{eq-ch8bounds}\bar U_{\xi_a} (y,l) \le u_\infty(y,l) \le \bar U_{\xi_b} (y,l).\ee $g_\infty(l) = u^{1-m}_\infty(y,l)\delta_{ij} $ is an eternal solution of the flow which has nonnegative Ricci curvature.

To apply Proposition \ref{prop-rigidity} to $u_\infty(y,l)$, we need to show it has actually strictly positive Ricci curvature and uniformly bounded $|Rm|$. 
 We first show uniform boundedness of curvature.  By \eqref{eq-114}, $\bar g_\infty(l)$ has  bounded curvature on $\R^n\setminus B_{r_1}(0)$,  for some  large $r_1$. We also have a uniform curvature bound of $\bar g_{\infty}(l)$ on $B_{r_1}(0)$ by two bounds \eqref{eq-asympbounds} and interior uniformly parabolic regularity estimate of the equation \eqref{eq-Einfty}. Since $\bar g_\infty(l)$ and $g_\infty$ are isometric, this gives uniform bound of $|\text{Rm}|$. Next, the proof for positive Ricci uses Theorem \ref{thm-classification}, the classification of locally conformally flat  nonnegative Ricci Yamabe flow having a nontrivial null eigenvector. It  solely an  interesting result, so we prove it  in a separate theorem. $(\R^n,\bar g_{\infty}(l))$ can not be flat by bounds \eqref{eq-ch8bounds}. Also an eternal solution can not be isometric to a cylinder solution which exists up to a finite time. Hence Ricci of $\bar g_\infty(l)$ is positive definite everywhere by Theorem \ref{thm-classification}.
 
Finally, by Proposition \ref{prop-rigidity}, we conclude $u_\infty(y,l) = U_{\xi_0}(y,l)$ for some $\xi_0\le \xi_a$.

\end{proof}

We will finish with proving  of the following result which was used above in the proof of Theorem \ref{thm-main}. 

\begin{theorem}\label{thm-classification}For $n\ge3$, let $(M,g(t))$ for $t\in(0,T)$ be a complete locally conformally flat solution of the Yamabe flow which has nonnegative Ricci and uniformly bounded Riemann curvature. If the  Ricci tensor has a null eigenvector at some point $(p_0,t_0)$, then $(M,g(t))$ is either locally isometric to  flat Euclidean space or  a cylinder solution $(\R\times S^{n-1}, f(t)(dr^2 \times g_{\text{can}}))$ where $g_{\text{can}}$ is the  round metric on $S^{n-1}$ and $f(t)=(n-1)(n-2)(T'-t)$ for some $T'>T$. 

\begin{proof} The uniform boundedness of the Riemann curvature tensor will only be used to apply the  (strong) maximum principle. For a  locally  conformally flat solution of the Yamabe flow, the evolution of Ricci tensor $R_{ij}$ is shown in Lemma 2.4 \cite{Ch}  as \[ \partial_t R_{ij} = (n-1)\Delta R_{ij} + \fr{1}{n-2} B_{ij} \] where $B_{ij}$ is a quadratic expression of $R_{ij}$. It was shown in (2.11) and (2.12) of \cite{Ch} that, with respect to an orthonormal basis which diagonalize Ricci tensor by $R_{ij}=\lambda_i\delta_{ij}$, we have $B_{ij}= \mu_i \delta_{ij}$ where \be \label{eq-bij}\mu_i = \sum_{k,l\neq i, \, k>l}(\lambda_k-\lambda_l)^2+ (n-2)\sum_{k\neq i} (\lambda_k-\lambda_i)\lambda_i.\ee

Let $\lambda_1\le \lambda_2\le \ldots \le \lambda_n$  be the eigenvalues of $R_{ij}$ in an increasing order. Note that
for any $1 \leq k \leq n$, we have 
$$m_k:=\lambda_1+\lambda_2+\cdots +\lambda_k = \inf \{ \text{Tr}_g (R_{ij}(V,V))\,\, |\, V\subset T_pM\text{ is a subspace of dim k}  \} $$ 
is a concave function of $R_{ij}$. 
Since the solution has  nonnegative Ricci, $m_j=0$ implies $\lambda_i=0$ for all $i\le j$. From equation  \eqref{eq-bij}, it is easy to check that the ODE $\partial_t R_{ij} = B_{ij}$ preserves $m_k\ge0$ under the nonnegative Ricci condition.  Therefore, we can apply the strong maximum principle (Lemma 8.1 in \cite{Ha2}) on $m_k\ge0$. The lemma and the continuity of $m_k$ imply that either $m_k\equiv0$ or $m_k>0$ everywhere at each time $t=t'$. Furthermore, if $m_k>0$ at $t=t'$, $m_k>0$ for all $t>t'$. As a consequence, there is a well defined decreasing function $\hat k(t)\in\{0,\ldots, n\}$ such that $m_k(p,t)=0$ if $k\le \hat k(t)$ and $m_k(p,t)>0$ if $k>\hat k(t)$.  Since $m_k=0$ iff $\text{dim} (\text{Null}(R_{ij}))\ge k$, we conclude that the rank of $R_{ij}$ is constant in space and it is equal to $n-\hat k(t)$, which is increasing with respect to time.
 
 Under the assumption that there is a point $(p_0,t_0)$ where Ricci curvature has a null eigenvector, we will show that the rank of  Ricci
 curvature  is either $0$  or  $n-1$ for all time. By the previous argument, the Ricci tensor can't have  full rank for $t\le t_0$. Also since it is increasing, there is an interval of time $(t_1,t_2)$ with $t_2\le t_0$ such that $\text{dim}(\text{Null}(R_{ij}))=k$,  for some fixed $k\in\{1,\ldots,n-1,n\} $ on this  time interval. If $k=0$, then it is clear  that the solution must  be stationary for all time and the solution must be Ricci flat. Since on a locally conformally flat manifold the Riemann curvature tensor is determined by the Ricci tensor, this implies that the solution is locally euclidean. Next, in case where  $1\le k \le n-1$  we can exactly follow the argument of Lemma 8.2 \cite{Ha2}  on the time interval $(t_1,t_2)$ to conclude that the null space of the Ricci tensor is invariant under parallel translation and also it is invariant in time.  Moreover, it lies in the null space of $B_{ij}$. 
By this last property and \eqref{eq-bij}, we see that $k$ has to be $1$ and other $\lambda_i$s except $\lambda_1$ should be the same positive number (possibly different at each point).  In this case, the manifold locally splits off along this parallel 1-dimensional null eigenvector distribution (see the lemma which follows after Theorem 8.3 \cite{Ha2})  i.e.  $(M,g(t))$ is locally splits $(\R\times N^{n-1}, dr^2 \times g^N(t))$ where $(N,g^N(t))$ is a  solution of $n-1$ dimension the Yamabe flow.

Actually, it is locally isometric to a cylinder $(\R\times S^{n-1}, dr^2 \times g_{\text{can}}(t))$ where $g_{\text{can}}(t)$ is a round metric on the sphere. Let us fix a time $t$. From the previous observation that the other $\lambda$s are the same, we know that $(N^{n-1}, g^N(t))$ is an Einstein manifold. i.e. $\text{Ric}^N(x) = \lambda(x)g^N$. If $n-1\ge3$, $\lambda\equiv\text{constant}$ could be seen by the contracted second Bianchi identity. $\nabla_jR^j_{i} = \fr{1}{2}\nabla_iR$ implies \[\nabla_i\lambda = \fr{n-1}{2}\nabla_i\lambda \quad\text{ or } \quad0=\fr{n-1}{2}\nabla_i\lambda\] depending on the direction  $i$. When $(M^n,g(t))$ is locally conformally flat and $(N^{n-1},g^N(t))$ is Einstein, we directly check from the Weyl tensor of $(M,g)$ that $(N,g^N)$ is also locally conformally flat and a space form of positive sectional curvature. When $n=3$, the Cotton tensor of $(M,g)$ vanishes. \[C_3:=C_{ijk}=\nabla_i R_{jk}-\nabla_{j}R_{ik} -\fr{1}{4}\left(\nabla_iRg_{jk}-\nabla_{j}Rg_{ik}\right)\equiv0.\] This implies \[\nabla_i\lambda g_{jk}-\nabla_j \lambda g_{ik}=0\quad\text{ and hence }\quad g^{ik}(\nabla_i\lambda g_{jk}-\nabla_j \lambda g_{ik})=2\nabla_i\lambda=0.\] 
Now again $\lambda$ is a positive constant and this proves the theorem. 
\end{proof}\end{theorem}

\begin{remark} In addition to this, if the manifold is simply connected, the solution is globally $(\R^n,g_{\text{can}} )$ or $(\R\times S^{n-1}, f(t)(dr^2\times g_{\text{can}}))$ with $f(t)=(n-1)(n-2)(T'-t)$ for some $T'>T$. The only simply connected complete locally euclidean manifold is $(\R^n,g_{\text{can}} )$. When it locally splits off, let us consider a smooth unit null eigenvector field of  Ricci.  Its dual 1-form is closed since the vector field is parallel. Since the manifold is simply connected, it is (globally) exact and the potential function will give a global splitting. \end{remark}

\centerline{\bf Acknowledgements} The authors are indebted  to Jim Isenberg and Mariel Saez for useful   discussions  about
this work. 

\smallskip 
\noindent  Beomjun Choi has been partially supported by NSF grant  DMS-1600658.\\
Panagiota Daskalopoulos  has been partially supported by NSF grant   DMS-1600658 and the Simons Foundation.

\end{document}

%% file: figure0.pdf_tex
\begingroup%
  \makeatletter%
  \providecommand\color[2][]{%
    \errmessage{(Inkscape) Color is used for the text in Inkscape, but the package 'color.sty' is not loaded}%
    \renewcommand\color[2][]{}%
  }%
  \providecommand\transparent[1]{%
    \errmessage{(Inkscape) Transparency is used (non-zero) for the text in Inkscape, but the package 'transparent.sty' is not loaded}%
    \renewcommand\transparent[1]{}%
  }%
  \providecommand\rotatebox[2]{#2}%
  \ifx\svgwidth\undefined%
    \setlength{\unitlength}{628.4138041bp}%
    \ifx\svgscale\undefined%
      \relax%
    \else%
      \setlength{\unitlength}{\unitlength * \real{\svgscale}}%
    \fi%
  \else%
    \setlength{\unitlength}{\svgwidth}%
  \fi%
  \global\let\svgwidth\undefined%
  \global\let\svgscale\undefined%
  \makeatother%
  \begin{picture}(1,0.29476696)%
    \put(0,0){\includegraphics[width=\unitlength,page=1]{figure0.pdf}}%
    \put(0.94487998,0.04292564){\color[rgb]{0,0,0}\makebox(0,0)[lb]{\smash{$\eta$}}}%
    \put(0.27934775,0.27686145){\color[rgb]{0,0,0}\makebox(0,0)[lb]{\smash{$\hat w$}}}%
    \put(0,0){\includegraphics[width=\unitlength,page=2]{figure0.pdf}}%
    \put(0.31990581,0.22655483){\color[rgb]{0,0,0}\makebox(0,0)[lb]{\smash{$(n-1)(n-2)$}}}%
    \put(0.40166042,0.04101949){\color[rgb]{0,0,0}\makebox(0,0)[lb]{\smash{$A=\kappa^\gamma$}}}%
    \put(0,0){\includegraphics[width=\unitlength,page=3]{figure0.pdf}}%
    \put(0.1833577,0.01040299){\makebox(0,0)[lb]{\smash{Inner Region}}}%
    \put(0,0){\includegraphics[width=\unitlength,page=4]{figure0.pdf}}%
    \put(0.68291793,0.01094757){\color[rgb]{0,0,0}\makebox(0,0)[lb]{\smash{Outer Region}}}%
  \end{picture}%
\endgroup%

%% file: figure1.pdf_tex
\begingroup%
  \makeatletter%
  \providecommand\color[2][]{%
    \errmessage{(Inkscape) Color is used for the text in Inkscape, but the package 'color.sty' is not loaded}%
    \renewcommand\color[2][]{}%
  }%
  \providecommand\transparent[1]{%
    \errmessage{(Inkscape) Transparency is used (non-zero) for the text in Inkscape, but the package 'transparent.sty' is not loaded}%
    \renewcommand\transparent[1]{}%
  }%
  \providecommand\rotatebox[2]{#2}%
  \ifx\svgwidth\undefined%
    \setlength{\unitlength}{841.88976378bp}%
    \ifx\svgscale\undefined%
      \relax%
    \else%
      \setlength{\unitlength}{\unitlength * \real{\svgscale}}%
    \fi%
  \else%
    \setlength{\unitlength}{\svgwidth}%
  \fi%
  \global\let\svgwidth\undefined%
  \global\let\svgscale\undefined%
  \makeatother%
  \begin{picture}(1,0.70707071)%
    \put(0,0){\includegraphics[width=\unitlength,page=1]{figure1.pdf}}%
    \put(0.76703538,0.07569613){\color[rgb]{0,0,0}\makebox(0,0)[lb]{\smash{$\xi$}}}%
    \put(0.3201512,0.50578786){\color[rgb]{0,0,0}\makebox(0,0)[lb]{\smash{${\tiny \bar w}$}}}%
    \put(0.48546598,0.07685385){\color[rgb]{0,0,0}\makebox(0,0)[lb]{\smash{$\xi_1$}}}%
    \put(0,0){\includegraphics[width=\unitlength,page=2]{figure1.pdf}}%
  \end{picture}%
\endgroup%